\documentclass[11pt]{amsart}
\usepackage[usenames,dvipsnames]{color}
\usepackage[table]{xcolor}
\usepackage{lscape}
\usepackage{amsfonts,epsfig}
\usepackage{latexsym}
\usepackage{amssymb}
\usepackage{amsmath}
\usepackage{amsthm}
\usepackage{graphics}
\usepackage[all]{xy}
\usepackage[T2A]{fontenc}
\usepackage{multirow}
\usepackage{array,booktabs,ctable}
\usepackage{enumerate}
\usepackage{booktabs}
\usepackage{longtable}
\usepackage{color}
\usepackage{multirow}
\usepackage{caption}

\newcommand{\smallmat}[4]{\left(\begin{smallmatrix}#1&#2\\#3&#4\end{smallmatrix}\right)}

\usepackage{refcheck}
\usepackage[pagebackref]{hyperref}
\usepackage{xcolor,colortbl}

\renewcommand{\arraystretch}{1.3}

\newtheorem{theorem}{Theorem}
\newtheorem{corollary}[theorem]{Corollary}
\newtheorem{conjecture}[theorem]{Conjecture}
\newtheorem{lemma}[theorem]{Lemma}

\newtheorem{proposition}[theorem]{Proposition}
\newtheorem{definition}[theorem]{Definition}

\theoremstyle{definition}
\newtheorem{remark}[theorem]{Remark}
\newtheorem{example}[theorem]{Example}

\definecolor{light-gray}{gray}{0.1}

\DeclareMathOperator{\val}{val} 

\DeclareMathOperator{\tors}{tors}

\newcommand{\Q}{\mathbb Q}
\newcommand{\Qbar}{{\overline{\mathbb Q}}} 
\newcommand{\Z}{\mathbb Z}

\newcommand{\F}{\mathbb F}

\newcommand{\PP}{\mathbb P}

\newcommand{\SL}{\operatorname{SL}}

\newcommand{\Gal}{\operatorname{Gal}}
\newcommand{\Aut}{\operatorname{Aut}}

\newcommand{\GL}{\operatorname{GL}}
\newcommand{\PGL}{\operatorname{PGL}}

\definecolor{light-gray}{gray}{0.95}
\definecolor{light-gray2}{gray}{0.65}
\newenvironment{romanenum}{\hfill \begin{enumerate} }{\end{enumerate}}

\usepackage[alphabetic,backrefs,lite]{amsrefs} 

\begin{document}

\bibliographystyle{plain}
\title{Serre's constant of elliptic curves over the rationals}

\author{Harris B. Daniels}
\address{Department of Mathematics and Statistics, Amherst College, MA 01002, USA}
\email{hdaniels@amherst.edu} 
\urladdr{http://hdaniels.people.amherst.edu}

\author{Enrique Gonz\'alez-Jim\'enez}
\address{Universidad Aut{\'o}noma de Madrid, Departamento de Matem{\'a}ticas, Madrid, Spain}
\email{enrique.gonzalez.jimenez@uam.es}
\urladdr{http://matematicas.uam.es/~enrique.gonzalez.jimenez}


\thanks{The second author was partially  supported by the grant PGC2018--095392--B--I00.}

\date{\today}

\begin{abstract} 
Let $E$ be an elliptic curve without complex multiplication defined over the rationals. The purpose of this article is to define a positive integer $A(E)$, that we call the {\it Serre's constant associated to $E$}, that gives necessary conditions to conclude that $\rho_{E,m}$, the mod m Galois representation associated to $E$, is non-surjective. In particular, if there exists a prime factor $p$ of $m$ satisfying $\val_p(m) \ge \val_p(A(E))>0$ then $\rho_{E,m}$ is non-surjective. {Conditionally under Serre's Uniformity Conjecture, w}e determine all the Serre's constants of elliptic curves without complex multiplication over the rationals that occur infinitely often. 
Moreover, we give all the possible combination of mod $p$ Galois representations that occur for infinitely many non-isomorphic classes of non-CM elliptic curves over $\Q$, and the known cases that appear only finitely. We obtain similar results for the possible combination of maximal non-surjective subgroups of $\GL_2(\Z_p)$. Finally, we conjecture all the possibilities of these combinations and in particular all the possibilities of these Serre's constants.
\end{abstract}

\maketitle

\date{\today}

\section{Introduction}\label{sec:intro}
Let $E/\Q$ be an elliptic curve and $n$ a positive integer. We denote by $E[n]$ the $n$-torsion subgroup of $E(\Qbar)$, where $\Qbar$ is a fixed algebraic closure of $\Q$. The absolute Galois group $\Gal(\Qbar/\Q)$ acts on $E[n]$ by its action on the coordinates of the points, inducing a Galois representation  
$$
\rho_{E,n}\,:\,\Gal(\Qbar/\Q)\longrightarrow \Aut(E[n]),
$$
called the {\it mod $n$ Galois representation associated to $E$}.  Notice that since $E[n]$ is a free $\Z/n\Z$-module of rank $2$, fixing a $\Z$-basis of $E[n]$, we identify $ \Aut(E[n])$ with $\GL_2(\Z/n\Z)$. Then we rewrite the above Galois representation as
$$
\rho_{E,n}\,:\,\Gal(\Qbar/\Q)\longrightarrow\GL_2(\Z/n\Z).
$$
Therefore we can view $\rho_{E,n}(\Gal(\Qbar/\Q))$ as a subgroup of $\GL_2(\Z/n\Z)$, determined uniquely up to conjugacy, and denoted by $G_E(n)$ in the sequel. 

Fixing a prime $p$ and choosing compatible bases for $E[p^k]$ for all $k$, one can take the inverse limit of these mod $p^k$ Galois representations and construct a new map
\[
\rho_{E,p^\infty}: \Gal(\Qbar/\Q) \to \GL_2(\Z_p),
\]
called the {\it $p$-adic Galois representation associated to $E$}.  Let us denote by $G_E^\infty(p)$ the image of $\rho_{E,p^\infty}$, which is determined uniquely up to conjugacy. In certain instances we will say that $G_E(n)=G$ or $G^\infty_E(p)=G$ for some group $G$; this will always mean that we have fixed a basis so that we remove ambiguity of working up to conjugacy.

Suppose that $E$ does not have complex multiplication (CM in the sequel). One of the first major results about the images of Galois representations associated to an elliptic curve is a renowned theorem of Serre \cite[Th\'eor\`eme 2]{Serre} that asserts that $\rho_{E,p}$ is not surjective for a finite number of primes $p$, called exceptional primes (Duke \cite{duke} showed that almost all non-CM elliptic curves have no exceptional primes). In other words, there exists a positive integer $C_E$, depending on $E$, such that $\rho_{E,p}$ is surjective for any prime $p> C_E$. After proving this theorem, Serre immediately asked  \cite[\S 4.3]{Serre} if it was possible to make the constant $C_E$ independent of $E$. Moreover, Serre in \cite[page 399]{serre81} asked if $C_E< 41$ always holds. That is, he asked the following question:\\

{\noindent \bf Serre's Uniformity Question.} {\em If $E/\Q$ is a non-CM elliptic curve, then must it be that $\rho_{E,p}$ is surjective for any prime $p\ge 41$?}\\

Nowadays, {an affirmative answer to the above question has received the name of {\it Serre's Uniformity Conjecture} (or sometimes just {\it Uniformity Conjecture})} despite the fact that Serre himself never conjectured it to be true.  Since Serre first asked the question there has been much progress towards to proving it. A summary of these results can be found in the following theorem (see \cite{curse, bilu, pbr, Mazur1978,Serre, zywina1} for more details):

\begin{theorem}\label{T:Mazur-Serre-BPR}
Let $E/\Q$ be a non-CM elliptic curve and $p$ a prime. Then one the following possibilities occurs:
\begin{romanenum}
\item $G_E(p)=\GL_2(\F_p)$.
\item $p\in  \{2, 3, 5, 7, 11, 13, 17, 37\}$, and $G_E(p)$ is conjugate in $\GL_2(\F_p)$ to one of the groups in Tables \ref{jmaps} and \ref{table_mod_single}.
\item\label{iii} $p\geq 17$: $G_E(p)$ is conjugate to a subgroup of the normalizer of a non-split Cartan subgroup of  level $p$. 
\end{romanenum}
\end{theorem}

{ Beside the above theorem, Lemos (\cite[Theorem 1.1]{lemos1}, \cite[Theorem 1.4]{lemos2}) has recently obtained partial results in the direction of a complete proof of the Serre's Uniformity Conjecture:

\begin{theorem}\label{lemos}
Let $E/\Q$ be a non-CM elliptic curve. Suppose that one the following possibilities occurs:
\begin{romanenum}
\item\label{lemos1} $E$ admits a non-trivial cyclic isogeny defined over $\Q$.
\item\label{lemos2} There exists a prime $q$ for which $G_E(q)$ is contained in the normalizer of a split Cartan subgroup of $\GL_2(\F_q)$.
\end{romanenum}
Then $\rho_{E,p}$ is surjective for any prime $p>37$.
\end{theorem}
}

{Zywina conjectures  \cite[Conjecture 1.12]{zywina1} that Theorem \ref{T:Mazur-Serre-BPR} (\ref{iii}) is not possible. This is what we call the Strong Uniformity Conjecture.

\

{\noindent \bf Strong Uniformity Conjecture.} {\em Let $E$ be a non-CM elliptic curve defined over $\Q$ and $j_E$ its $j$-invariant. If $p\geq 17$ is a prime such that 
$$
(p, j_E)\not\in \left\{(17,-17\cdot 373^3/2^{17}), (17,-17^2\cdot 101^3/2), (37,-7\cdot 11^3), (37, -7\cdot 137^3\cdot2083^3)  \right \},
$$
then $G_E(p)=\GL_2(\F_p)$.
}

\

Zywina classifies all of the possible images of the mod $p$ images of non-CM elliptic curves defined over $\Q$ given the corresponding moduli spaces for $p\le 13$, except the case when the image of $G_E(13)$ in $\PGL_2(\F_{13})$ is isomorphic to $S_4$ (the permutation group of 4 elements) (see \cite{zywina1,curse,baran14}). 

\begin{conjecture}\label{13S4}
 Let $E$ be a non-CM elliptic curve defined over $\Q$ and $j_E$ its $j$-invariant. Then the image of $G_E(13)$ in $\PGL_2(\F_{13})$ is isomorphic to $S_4$ if and only if 
$$
j_E\in \left \{\frac{2^4\cdot 5\cdot 13^4\cdot 17^3}{3^{13}},-\frac{2^{12}\cdot 5^3\cdot 11\cdot 13^4}{3^{13}} \,,\,\frac{2^{18}\cdot3^3\cdot 13^4\cdot 127^3\cdot 139^3\cdot 157^3\cdot 283^3\cdot 929}{5^{13}\cdot 61^{13}}\right\}.
$$
\end{conjecture}

Assuming\footnote{J.~S. Balakrishnan, N. Dogra, J.~S. M\"uller, J. Tuitman, and J. Vonk have recently announced a proof of this conjecture but have yet to make the results available publicly.} Conjecture \ref{13S4} and the Strong Uniformity Conjecture, Zywina gives all the possible groups $G_E(p)$ for $E/\Q$ and the corresponding moduli spaces. This data can be found in Tables \ref{jmaps} and \ref{table_mod_single}. 

\begin{remark}\label{remadic}
Let $E/\Q$ be a non-CM elliptic curve and $p$ a prime. Serre  \cite[IV]{serre2} showed that if $p\geq 5$ then $\rho_{E,p}$ is surjective if and only if  $\rho_{E,p^\infty}$ is surjective. But when $p = 2$  or $3$ it is not the case. The reason is that there are proper subgroups of $\SL_2(\Z/4\Z)$ and $\SL_2(\Z/8\Z)$ that surject onto $\SL_2(\Z/2\Z)$ under the standard reduction map as well as a proper subgroup of $\SL_2(\Z/9\Z)$ that surjects onto $\SL_2(\Z/3\Z)$. These groups and the corresponding moduli spaces of elliptic curves can be found in \cite{elkies2006elliptic} or \cite{mavrides} and \cite{dok^2} or \cite{RZB} and are available in Table \ref{tab:adic_groups}. 
\end{remark}

In view of the above remark we make the following definition:

\begin{definition}
Let $E/\Q$ be a  non-CM elliptic curve and $p$ a prime. We say that $p$ is {\it adically-exceptional} for $E$ if $\rho_{E,p^\infty}$ is not surjective. 
\end{definition}
 
Remark \ref{remadic} asserts that if $p\ge 5$ then $p$ is exceptional if and only if it is adically-exceptional, and that the only possible primes that could be non-exceptional but adically-exceptional are $p=2$ and $p=3$.

An affirmative answer to the Serre's uniformity question does not give any information about the possible combinations of exceptional primes (or adically-exceptional primes) which may occur for a given non-CM elliptic curve defined over $\Q$. In attempt to study this question, we give the following definition\footnote{The definition of Serre's constant that appears in this paper is a generalization of the one that Cojocaru defined at \cite{cojocaru}.}.

\begin{definition}
Let $E/\Q$ be a non-CM elliptic curve, then we define {\it Serre's constant} associated to $E$ to be 
\[
A(E) = \prod_{p\,\,\text{prime}} p^k
\]
where $k$ is the smallest positive integer such that $\rho_{E,p^k}$ is non-surjective if such an integer exists or $0$ otherwise.
\end{definition}

Note that if $p$ is not adically-exceptional then $\val_p(A(E))=0$. In particular, if no prime is not adically-exceptional for $E$ then $A(E)=1$. Moreover, Jones \cite{jones} has proved that almost all non-CM elliptic curve over $\Q$ have $A(E)=1$. On the other hand, if $m$ is a positive integer, $A(E)$ gives necessary conditions to conclude that $\rho_{E,m}$ is non-surjective.

\begin{proposition}
Let $E/\Q$ be a non-CM elliptic curve and $m\in\mathbb N$. If there exists a prime factor $p$ of $m$ satisfying $\val_p(m) \ge \val_p(A(E))>0$ then $\rho_{E,m}$ is non-surjective. 
\end{proposition}

\begin{remark}
The hypothesis of above proposition is a neccessary  but not sufficient condition. Let $E/\Q$ be the elliptic curve with Cremona label \texttt{3891b1}. Then $E$ is a Serre curve, so all the $p$-adic Galois representations are surjective. Therefore $A(E)=1$. In this case if $m=2\cdot 3\cdot 1297$, then the mod $m$ Galois representation is non-surjective. The reason is that this curve has entanglement: $\Q(E[2]) \cap \Q(E[3891]) = \Q(\sqrt{\Delta_E})$, where $\Delta_E$ is the discriminant of the minimal model of $E$, and $[\GL_2(\Z/m\Z):G_E(m)]=2$ (cf. \cite{harris}).
\end{remark}

We will use the following notation:
\begin{itemize} 
\item Let $\mathcal A$ be the set of the integers $A(E)$ where $E$ runs over all non-CM elliptic curve over $\Q$.
\item  Let $\mathcal{A}_\infty$ be the subset of $\mathcal A$ that occur infinitely often. More precisely, $N\in \mathcal{A}_\infty$ if there are infinitely many non-CM elliptic curves $E$, non-isomorphic over $\Qbar$, such that $A(E)=N$.
\end{itemize}

\begin{remark}
Notice that Serre's Uniformity Conjecture is equivalent to the finiteness of the set $\mathcal A$. { We also point out here that the exponents that appear on 2 and 3 on numbers in $\mathcal{A}$ are bounded. We know this because, if $E/\Q$ is an elliptic curve, then if $\rho_{E,8}$ or $\rho_{E,9}$ respectively are surjective, then $\rho_{E,2^\infty}$ or $\rho_{E,3^\infty}$ respectively have to be surjective.}
\end{remark}

The first theorem of this paper is the following:

\begin{theorem}\label{thm_serre_constant}  
Assuming the Uniformity Conjecture: 
$$
\mathcal{A}_\infty = \{1,2,3,4,5,6,7,8,9,10,11,12,13,14,15,20,21,24,28,40,56,104\}.
$$
\end{theorem}

\begin{remark}

We need to assume the Uniformity Conjecture in the proof of the complete classification of $\mathcal{A}_\infty$ at Theorem \ref{thm_serre_constant} in order to ensure that there are infinitely many non-isomorphic curves with the Serre's constant that we want. Without assuming uniformity we can not be sure that for all but finitely many of the $\Qbar$-isomorphism classes the associated Serre's constant does not contain some unexpected prime factors. Using Theorem \ref{lemos}, much of Theorem \ref{thm_serre_constant} can be made independent of the Uniformity Conjecture. In fact, only $4,8,9,11,20$ and $21$ are conditionally under uniformity. 
\end{remark}

In Section \ref{sec:maintheorems}, Theorem \ref{the_comb_modp} classifies the possible combinations of mod $p$ Galois representations that occur for infinitely many non-isomorphic classes of non-CM elliptic curves over $\Q$, and the known cases that appear only finitely. We obtain similar results in Theorem \ref{the_comb_adic} for the possible combinations of maximal non-surjective subgroups of $\GL_2(\Z_p)$. Theorem \ref{thm_serre_constant}   is a direct consequence of the Theorem \ref{the_comb_modp}  and  \ref{the_comb_adic}. Besides classifying which numbers occur as Serre's constant for infinitely many elliptic curve we construct the moduli space for each of the possible combinations of images and determine all of the points on the corresponding modular curves when the genus $\le 2$. Further, for each curve of genus $\geq 3$ we do a point search to find all easily visible points and we compute Serre's constant for each curve in the LMFDB. The resulting data is compiled in tables at the Appendix. These results and the previous search motivate the following conjecture. 

\begin{conjecture}\label{conj:A}
$\mathcal{A}=\mathcal{A}_{\infty}\cup \{17,36,37,44,60,120,168\}$.
\end{conjecture}

The above conjecture is being treated in an ongoing continuation of this paper at \cite{DGJ2}.

\

\noindent{\textit{\bf Acknowledgements.} } {The authors would like to thank Alina Cojocaru, \'Alvaro Lozano-Robledo, Filip Najman, Andrew Sutherland, Xavier Xarles {and David Zywina} for help in the preparation of this article. We would like to thank to John Cremona for providing access to computer facilities on the Number Theory Warwick Grid at University of Warwick, where the main part of the computations where done. The authors would also like to thank the anonymous referee for useful comments during the review process as well as the editors of this paper for a speedy review.

\

\noindent{\textit{Notation.} } Through out the paper we will refer to conjugacy classes of subgroups of $\GL_2(\Z/p\Z)$ using the notation established by Sutherland in \cite[Section 6.4]{Sutherland2} and used throughout the LMFDB database \cite{lmfdb}. Notice that Zywina \cite{zywina1} uses different notation for such conjugacy classes and in Tables \ref{jmaps} and \ref{table_mod_single} we give the translation between Sutherland's and Zywina's labels. Any specific elliptic curves mentioned in this paper will be referred to by Cremona label \cite{cremonaweb,antwerp} and a link to the corresponding LMFDB page \cite{lmfdb} for the ease of the reader. 

\section{Results for Combinations of Galois representations for non-CM elliptic curves over $\Q$}\label{sec:maintheorems}
One of the goals of this paper is to classify all the possible combinations of mod $p$ Galois representations attached to elliptic curves defined over $\Q$. We wish to point out here that Morrow \cite{morrow} began the study of the possible combinations of mod $n_1n_2$ Galois representations such that $n_1$ is a power of 2 and $n_2<17$ is a prime. Then Camacho-Navarro et al.~\cite{camacho} are continuing this study to the case of subgroups of $\GL_2(\Z/n_1n_2\Z)$ where the corresponding modular curve has low genus, and/or is hyperelliptic. 

In order to establish the appropriate language to study the possible combinations of images that can occur we give the following definitions.

\begin{definition}
Let $E/\Q$ be a non-CM elliptic curve and $S_E$ be the set of exceptional primes of $E$. Let $S\subseteq S_E$ and for each $p\in S$ let $G_p$ be a proper subgroup of $\GL_2(\Z/p\Z)$. We say that $E$ is of {\em exceptional type} (or {\em type} for short) $[G_p : p \in S]$ if for every $p\in S$ the group $G_E(p)$ is conjugate to a subgroup of $G_p$. We say that the {\em exact exceptional type} (or {\em exact type}) of $E$ is $[G_p : p \in S]$ if $S = S_E$ and $G_E(p)$ is conjugate to $G_p$ (not a proper subgroup of $G_p$) for every $p\in S$. 
\end{definition}

Here we consider two possible types $[G_p : p \in S]$ and $[H_p : p \in T]$ equal if $S=T$ and for every $p\in S$, $G_p$ is conjugate to $H_p$ in $\GL_2(\Z/p\Z)$. Similarly, we say that $[G_p : p \in S]$ is a {\em smaller type} than (or {\em subtype of}) $[H_p : p \in T]$ if $S\subseteq T$ and $G_p$ is conjugate to a subgroup of $H_p$ for every $p\in S$. We will refer to $\#S$ as the {\em length} of type $[G_p : p \in S]$. With these conventions, the exact type of $E/\Q$ is unique and equal to $[G_E(p): p \in S_E]$. We also define the level of a given type $[G_p : p \in S]$ to be  $\prod_{p\in S} p.$ We say a type $[G_p : p \in S]$ is {\em maximal} if it is not a subtype of any other type of the same level. 
We point here out that for almost all elliptic curves $E/\Q$ the set $S_E = \emptyset$ (cf. \cite{duke}). In this case we say that the exact type of $E$ is $[\ ]$ and refer to this as the trivial type.

Before moving on we introduce the concept of a modular curve and explore the relationship between these curves and certain types. Let $G$ be a subgroup of $\GL_2(\Z/n\Z)$ satisfying $-I\in G$ and $\det(G)=(\Z/n\Z)^\times$. There is a {\it modular curve $X_G$} associated to $G$. This curve is defined over $\Q$, smooth, projective and geometrically irreducible. Moreover, there is a non-constant morphism $j_G\,:\,X_G\to \mathbb P^1(\Q)$, called the $j$-map of $G$. Moreover, given another group $G\subsetneq G'\subseteq \GL_2(\Z/n\Z)$ satisfying $-I\in G'$ and $\det(G')=(\Z/n\Z)^\times$, there exists a non-constant morphism $X_G\to X_{G'}$. One of the main properties of the pair $(X_G,j_G)$ is that for an elliptic curve $E/\Q$ with $j$-invariant $j_E\notin \{0,1728\}$, $G_E(n)$ is conjugate in $\GL_2(\Z/n\Z)$ to a subgroup of $G$ if and only if $j_E \in j_G(X_G(\Q))\cap \Q$.

{ \begin{remark}\label{rmk:twisting_and_-I}
Here we pause to point out that it is sufficient to consider $G$ containing $-I$ since if $H$ is an index 2 subgroup of $G$ such that $G = \langle H, -I \rangle$, then $X_G\simeq X_H$. This is because every non-cuspidal non-CM point on $X_G$ corresponds to a $\Qbar$-isomorphism class of elliptic curves such that $G_E(n)$ is conjugate to a subgroup of $G$. Inside of each of these $\Qbar$-isomorphism classes there is at least one $\Q$-isomorphism class (or twist) of curves whose image is actually contained in $H$. Thus, classifying the rational points on $X_H$ amounts to classifying the points on $X_G$ and then determining which twists in the $\Qbar$-isomorphism classes actually have $G_E(n)$ in $H$. {We also point out here for every elliptic curve with rational $j$-invariant, there is an elliptic curve with the same $j$-invariant such that $-I$ is in the image of the adelic Galois representation associated to the new elliptic curve.}
\end{remark}
}

The relationship between modular curves arising from the fact that given a type $[G_p : p \in S]$ of level $N$ we can associate a group $G\subseteq \GL_2(\Z/N\Z)$, where $G$ is the largest group (by containment) such that $\pi_p(G) = {\pm} G_p$ for all $p\in S$. Here $\pi_p: \GL_2(\Z/N\Z) \to \GL_2(\Z/p\Z)$ is the standard componentwise reduction map which is well-defined since $p$ divides $N$ by construction. { We point out here that the condition is that $\pi_p(G) = \pm G_p$ and not just that $\pi_p(G) = G_p$ so that we can ensure that $-I$ is in the group associated to a given type.} Then, associated to the group $G$ is the modular curve $X_G$ whose { non-cuspidal and non-CM} rational points correspond to $\Qbar$-isomorphism classes of elliptic curves over $\Q$ of type $[G_p : p \in S]$. 

{ In the other direction, given a group $G\subseteq \GL_2(\Z/N\Z)$ with $-I\in G$ and $\det(G) =(\Z/N\Z)^\times$, we let $S_G$ be the set of primes $p$ such that $p$ divides $N$ and $\pi_p(G)\neq \GL_2(\Z/p\Z)$. Then we can associate a type $[\pi_p(G) : p\in S_G]$ of level $\prod_{p\in S_G}(p)$. Again, the $\Qbar$-isomorphism classes of curves coming from points in $X_G(\Q)$ all have type $[\pi_p(G) : p\in S_G]$. In fact, every $\Qbar$-isomorphism classes of curves of type $[\pi_p(G) : S_G]$ arises as a point on $X_G$ exactly when $N$ is squarefree and $G$ is maximal among the groups of level $N$ corresponding to the type $[\pi_p(G) : p\in S_G]$.}

This association of types with groups, and hence modular curves, will be extremely useful in studying what combinations of images can occur. We will use it intimately in the remaining sections.  \\

Next we give a definition of {\em adic-type} and {\em exact adic-type}.

\begin{definition}
Let $E/\Q$ be a non-CM elliptic curve and $S_E^\infty$ be the set of adically exceptional primes. Let $S\subseteq S^\infty_E$. For each $p\in S$, let $G_p$ be a proper subgroup of $\GL_2(\Z_p)$.  We say that $E$ is of {\em adically-exceptional type} (or {\em adic-type} for short) $[G_p : p \in S]$ if for every $p\in S$ the group $G_E^\infty(p)$ is conjugate to a subgroup of $G_p$. We say that the {\em exact adically-exceptional type} (or {\em exact adic-type}) of $E$ is $[G_p : p \in S]$ if $S = S_E^\infty$ and $G_E^\infty(p)$ is conjugate to $G_p$ (not a proper subgroup group of $G_p$) for every $p\in S$. 
\end{definition}

We adopt similar conventions as above to compare two adic-types so that everything is well-defined changing what is necessary.  For a given adic-type $[G_p : p \in S]$ we define the {\em level} of that type to be $\prod_{p\in S} p^{k_p}$ where $k_p$ is the minimum integer such that the standard componentwise reduction map $G_p\to \GL_2(\Z/p^{k_p}\Z)$ is not surjective. 
For the sake of notational brevity, for each $G_p\subseteq \GL_2(\Z_p)$ that occurs in an adic-type we will denote $G_p$ by a subgroup $\tilde{G}_p\subseteq \GL_2(\Z/p^k\Z)$ for some $k$ such that $G_p = \pi^{-1}(\tilde{G}_p)$ where $\pi:\GL_2(\Z_p) \to \GL_2(\Z/p^k\Z)$ is again the standard componentwise reduction map. 

\begin{remark}
Let $E/\Q$ be an elliptic curve and $\rho_E:\Gal(\Qbar/\Q) \to \GL_2(\widehat{\Z})$ be the adelic Galois representation associated to $E$ constructed choosing bases for $E[n]$ compatible with divisibility and taking inverse limits. It is tempting to think that knowing the exact adic-type of $E$ is equivalent to knowing the image of $\rho_E$ upto conjugation, but this is not the case. The gap is that the exact adic-type does not contain any information about the entanglements between the field of definition of each Tate module. Serre showed that there must be {\em some} entanglement between these fields using the Weil paring and the Kronecker--Weber theorem and so we usually cannot recover the image of the adelic Galois representation attached to $E$ just from the exact adic-type of $E$.
\end{remark}

The following theorem gives the set of possible exceptional types that occur for infinitely many non-isomorphic classes of non-CM elliptic curves over $\Q$, and the known cases that appear only finitely.
\begin{theorem}\label{the_comb_modp}
{ (A) The following nontrivial exceptional types occur for infinitely many non-isomorphic classes of non-CM elliptic curves over $\Q$:
}
\begin{itemize}
\item  $[\, G_p\, ]$ for any $G_p$ in Tables \ref{jmaps} and \ref{table_mod_single} except {\rm\texttt{7Ns.3.1}, \texttt{11B.10.4}, \texttt{11B.10.5}, \texttt{13S4}, \texttt{17B.4.2}, \texttt{17B.4.6}, \texttt{37B.8.1}} and {\rm\texttt{37B.8.2}} (that appear only for finitely many $\Q$-isomorphic classes).
 \item $[\,G_2 \,,\,G_p \,]$, where $G_2$ is:
\begin{itemize}
\item[$\star$]  {\rm\texttt{2B}} and $G_p$ is {\rm\texttt{3B}, \texttt{3B.1.1}, \texttt{3B.1.2}, \texttt{3Cs}, \texttt{3Cs.1.1}, \texttt{3Nn}, \texttt{3Ns}, \texttt{5B},  \texttt{5B.4.1}, \texttt{5B.1.1},  \texttt{5B.1.4}, \texttt{5B.4.2}, \texttt{5B.1.2}}, or {\rm\texttt{5B.1.3}}.
\item[$\star$] {\rm\texttt{2Cn}} and $G_p$ is {\rm\texttt{3B}, \texttt{3B.1.1}, \texttt{3B.1.2}, \texttt{5S4}, \texttt{7B}, \texttt{7B.2.1}}, or {\rm\texttt{7B.2.3}}.
\item[$\star$] {\rm\texttt{2Cs}} and $G_p$ is  {\rm\texttt{3B}, \texttt{3B.1.1}}, or {\rm\texttt{3B.1.2}}.
\end{itemize}
\item  $[\,G_3\,,\,G_p\,]$, where $G_3$ is  {\rm\texttt{3Nn}} and  $G_p$ is {\rm\texttt{5B}}, {\rm\texttt{5Ns}}, {\rm\texttt{5Nn}}, or {\rm\texttt{7Nn}}.
\end{itemize}
(B) For the following exceptional types there are only a finite number of $\Qbar$-isomorphic classes:
\begin{itemize}
\item  $[\, G_p\, ]$ where $G_p$ is {\rm\texttt{7Ns.3.1}, \texttt{11B.10.4}, \texttt{11B.10.5}, \texttt{13S4}, \texttt{17B.4.2}, \texttt{17B.4.6}, \texttt{37B.8.1}} or {\rm\texttt{37B.8.2}}. 
\item $[\,G_3 \,,\,G_p\,]$, where $G_3$ is:
\begin{itemize}
\item[$\star$] {\rm\texttt{3B}} and  $G_p$ is {\rm\texttt{5B}, \texttt{5B.4.1}, \texttt{5B.1.1}, \texttt{5B.4.2}, \texttt{5B.1.2}, \texttt{5S4}, \texttt{7B}, \texttt{7B.2.1}}, or {\rm\texttt{7B.2.3}}.
\item[$\star$] {\rm\texttt{3B.1.1}} or {\rm\texttt{3B.1.2}} and $G_p$ is {\rm\texttt{5B.1.3}, \texttt{5B.1.4}, \texttt{5B.4.1}, \texttt{5B.4.2}, \texttt{5S4}, \texttt{7B}, \texttt{7B.2.1}}, or {\rm\texttt{7B.2.3}}.
\item[$\star$] {\rm\texttt{3Ns}} and  $G_p$ is {\rm\texttt{5B}}.
\end{itemize}
\end{itemize}
{Moreover, we give unconditionally the moduli space for each of the possible exceptional types (see Tables \ref{jmaps}, \ref{table_mod_single}, \ref{Eq_finemoduli}, \ref{jmapsj0coarse}, \ref{jmapsj0fine}, \ref{tab:gen1PosRank}, \ref{tab:gen1r0}, \ref{tab:gen1r0fine}), except for the cases of level $13$, $17$ and $37$ which are conditionally under Conjecture \ref{13S4} and the Strong Uniformity Conjecture.}
\end{theorem}

\begin{corollary}\label{the_comb_modp_COR} 
Assuming Serre's Uniformity Conjecture, the set of nontrivial exact types such that there exist infinitely many non-isomorphic classes of non-CM elliptic curves over $\Q$ of that type correspond to the cases that appear at Theorem \ref{the_comb_modp} (A).
\end{corollary}

The next result is similar to Theorem \ref{the_comb_modp} but we do not obtain a complete characterization, instead we obtain only maximal adically exceptional types that can occur.

\begin{theorem}\label{the_comb_adic}
{(A) The following list of maximal adically-exceptional types occur for infinitely many non-isomorphic classes of non-CM elliptic curves over $\Q$:
}
\begin{itemize}
\item $[\, G_p\, ]$ where $G_p$ is {\rm\texttt{2B}, \texttt{2Cn}, \texttt{3B}, \texttt{3Nn}, \texttt{5Nn}, \texttt{5B}, \texttt{5S4}, \texttt{7Ns}, \texttt{7Nn}, \texttt{7B}, \texttt{11Nn}} or {\rm\texttt{13B}}.
\item $[\, G_p\, ]$ where $G_p$ is {\rm\texttt{4X3}, \texttt{4X7}, \texttt{8X4}, \texttt{8X5}}, or  {\rm\texttt{9XE}}. 
\item $[\,G_2\,,\,G_p\,]$, where $G_2$ is:
\begin{itemize}
\item[$\star$]   {\rm\texttt{2B}} and $G_p$ is {\rm\texttt{3B}, \texttt{3Nn}}, or {\rm\texttt{5B}}.
\item[$\star$]   {\rm\texttt{2Cn}} and $G_p$ is {\rm\texttt{3B}, \texttt{5S4}}, or {\rm\texttt{7B}}.
\item[$\star$]   {\rm\texttt{4X3}} and $G_p$ is {\rm\texttt{3B}, \texttt{5S4}}, or {\rm\texttt{7B}}.
{
\item[$\star$]  {\rm\texttt{4X7}} and $G_p$ is {\rm\texttt{3Nn}}, or {\rm\texttt{5S4}}.
}
\item[$\star$]  {\rm\texttt{8X4}} and $G_p$ is {\rm\texttt{3B}, \texttt{5S4}, \texttt{7B}, \texttt{5B}}, or {\rm\texttt{13B}}.
{
\item[$\star$]  {\rm\texttt{8X5}} and $G_p$ is {\rm\texttt{3B}, \texttt{5S4}, {\rm\texttt{5Nn}}}, or {\rm\texttt{7B}}.
}
\end{itemize}

\item $[\,G_3\,,\,G_p\,]$, where $G_3$ is {\rm\texttt{3Nn}} and $G_p$ is  {\rm\texttt{5B}}, {\rm\texttt{5Nn}}, or {\rm\texttt{7Nn}}.
\end{itemize}
(B) There is only a finite number of $\Qbar$-isomorphic classes of elliptic curves with the following maximal adic-types:
\begin{itemize}
\item $[\, G_p \,]$ where $G_p$ is {{\rm\texttt{13S4}, \texttt{17B.4.2}, \texttt{17B.4.6}, \texttt{37B.8.1}} or {\rm\texttt{37B.8.2}}.}
\item $[\,G_2\,,\,G_p\,]$, where $G_2$ is:
\begin{itemize}
\item[$\star$]  {{\rm\texttt{4X3}} and $G_p$ is {\rm\texttt{11B.10.4}}, or {\rm\texttt{11B.10.5}}.}
\item[$\star$] {{\rm\texttt{4X7}} and $G_p$ is {\rm\rm\texttt{9XE}, \texttt{3B}, \texttt{5B}}, or {\rm \texttt{7B}}.}
\item[$\star$] {\rm\texttt{8X5}} and $G_7$ is {\rm\texttt{7Ns.3.1}}.
\end{itemize}
\item {$[\,G_3\,,\,G_p\,]$, where $G_3$ is {\rm\texttt{3B}} and $G_p$ is {\rm\texttt{5B}, \texttt{5S4}}, or {\rm\texttt{7B}}.}
\item $[\,G_2\,,\,G_3\,,\,G_5\,]$, where $G_2$ (resp. $G_3$, $G_5$) is {\rm\texttt{4X3}} (resp. {\rm\texttt{3B},  \texttt{5S4}}).
\item $[\,G_2\,,\,G_3\,,\,G_5\,]$, where $G_2$ (resp. $G_3$, $G_5$) is {\rm\texttt{8X4}} (resp. {\rm\texttt{3B},  \texttt{5B}}).
\item $[\,G_2\,,\,G_3\,,\,G_7\,]$, where $G_2$ (resp. $G_3$, $G_7$) is {\rm\texttt{8X4}} (resp. {\rm\texttt{3B},  \texttt{7B}}).
\end{itemize}
{Moreover, we give unconditionally the moduli space for each of the possible maximal adic-types above (see Tables \ref{jmaps}, \ref{tab:adic_groups}, \ref{tab:jconstant}, \ref{jmapsj0coarse}, \ref{tab:gen1PosRank}, \ref{tab:gen1r0} and Remark \ref{triples}), except (maybe) the case {\rm\texttt{[4X7,9XE]}} and for the cases of level $13$, $17$ and $37$ that are conditionally under Conjecture \ref{13S4} and the Strong Uniformity Conjecture.}
\end{theorem}


{
\begin{corollary}\label{the_comb_adic_COR} 
Assuming Serre's Uniformity Conjecture, any adically-exceptional type that has infinitely many elliptic curve over $\Q$ of exactly that type must be a subtype of one of those listed in the cases that appear at Theorem \ref{the_comb_adic} (A).
\end{corollary}}


\section{Invariance of the Serre's constant under quadratic twists}

The purpose of this section is to prove that $A(E)$ is an invariant of the isomorphism class of a non-CM elliptic curve.
\begin{proposition}\label{prop:surj_twist}
Let $E/\Q$ be a non-CM elliptic curve and let $p$ be a prime. If $E'/\Q$ is a quadratic twist of $E$, then ${\rho}_{E,p}$ is surjective if and only if ${\rho}_{E',p}$ is surjective. 
\end{proposition}

\begin{proof}
Notice that it is enough to show that if ${\rho}_{E,p}$ is surjective, then ${\rho}_{E',p}$ is surjective. Now, if $E'$ is a quadratic twist of $E$, then there is a square free $D\in\Z$ such that $E$ and $E'$ are isomorphic over $\Q(\sqrt{D})$. Assume that $G_E(p)=\GL_2(\Z/p\Z)$. Then, it must be that either $G_{E'}(p) = \GL_2(\Z/p\Z)$ or $\GL_2(\Z/p\Z) \simeq G_{E'}(p)\rtimes\Z/2\Z$. In the last case, we would have that $G_{E'}(p)$ is a subgroup of index 2 inside of $\GL_2(\Z/p\Z)$. We have that  $\det\rho_{E,p}$ is the cyclotomic character, a standard consequence of the Weil pairing says that $\det:G_{E'}(p)\to (\Z/p\Z)^\times$ is surjective. According to \cite[Figure 5.1]{adelmann} there are no such groups unless $p=2$ and $G_{E'}(p)$ is conjugate to \texttt{2Cn}, but the property that $G_{E'}(p)$ is conjugate to \texttt{2Cn} is equivalent to $E'$ having a square discriminant which is invariant under quadratic twists. This would mean that $G_E(p)$ is not surjective giving a contradiction. 
\end{proof}

\begin{corollary}\label{cor:only_depends_on_j}
Let $E/\Q$ be a non-CM elliptic curve. Then $A(E)$ is invariant under $\Qbar$-isomorphism. In other words, $A(E)$ only depends on $j_E$.
\end{corollary}

\begin{proof}All that is left is to prove is that if $p$ is adically-exceptional for $E$, then it is adically-exceptional with the same exponent for all of the quadratic twists of $E$. This follows from the exact same argument as above together with the fact that the $5$ maximal groups in Table \ref{tab:adic_groups} that surject onto $\GL_2(\Z/p\Z)$ but are not all $\GL_2(\Z/p^k\Z)$ for some $k\geq 2$, all contain $-I$ and thus can not be quadratic twisted into. That is, because they contain $-I$, if $G^\infty_E(p)$ is not in one of these groups and $E'$ is a quadratic twist of $E$, then $G^\infty_{E'}(p)$ is not in one of these groups either. 
\end{proof}

Corollary \ref{cor:only_depends_on_j} allow us to take representatives of each of the finitely many $\Qbar$-isomorphism classes of elliptic curves and compute their Serre's constant.

\section{Outline of the computations: Fiber products of pairs of modular curves}\label{outline}

We are now ready to fully leverage the connection between types and subgroups of $\GL_2(\Z/n\Z)$ and their corresponding modular curves. To start we survey some of the results about modular curves that will form the foundation for our computations. For $n=p\le 11$ prime, Zywina  \cite{zywina1} classifies all of the possible subgroups $G_E(p)\subseteq \GL_2(\Z/p\Z)$ for elliptic curves $E/\Q$. The case $p=13$ is the first prime for which Zywina does not have a complete description. He classifies all the possible subgroups $G_E(13)$ except the cases concerning subgroups of $\texttt{13Nn}$, $\texttt{13Ns}$ and $\texttt{13S4}$. Baran \cite{baran14} showed that the modular curves $X_\texttt{13Nn}$ and  $X_\texttt{13Ns}$ are both isomorphic to a genus 3 curve, and recently Balakrishnan et al. \cite{curse} have determined that this genus 3 curve has no non-singular, non-CM rational points. Therefore there are no non-CM elliptic curves $E/\Q$ such that $G_E(13)$ is a subgroup of $\texttt{13Nn}$ or $\texttt{13Ns}$. The remaining case is the curve $X_\texttt{13S4}$. Banwait and Cremona  \cite{BC14} have shown that this curve has genus 3 and at least three non-singular, non-CM points corresponding to the three $j$-invariants that appear in Conjecture \ref{13S4}. {Recently, Balakrishnan et al. have announced that using similar techniques to those in \cite{curse}, the curve $X_\texttt{13S4}$ has only the three non-singular, non-CM rational points found by Banwait and Cremona.} Finally, thanks to Theorem \ref{T:Mazur-Serre-BPR} we have that if $p\ge 17$ and $\rho_{E,p}$ is non-surjective then $G_E(p)$ appears in Table \ref{jmaps} or $G_E(p)$ is a subgroup of $\texttt{pNn}$. In the later case, Baran \cite{baran10} has showed that the genus of $X_\texttt{pNn}\ge 2$. 
On the other hand, by Remark \ref{remadic} if $p\ge 5$ we have $G_E^\infty(p) = \GL_2(\Z_p)$ if and only if $G_E(p) = \GL_2(\Z/p\Z)$ while when $p = 2$ or $3$, $G_E^\infty(p) = \GL_2(\Z_p)$ if and only if $G_E(p^k)$ is not conjugate to a subgroup of one of the groups listed in Table \ref{tab:adic_groups}.

The following table summarizes the genus of the modular curves of the form $X_G$ and whether or not the modular curve has infinitely many points: 

\begin{center}
\begin{tabular}{|c|c|c|}
\cline{2-3}
\multicolumn{1}{c|}{$\quad$} & Genus $X_G$ & $G=G^\infty_E(p)$\\
 \hline
\multirow{2}{*}{$\#X_G(\Q)<\infty$}  & > 1 & \texttt{13S4},\,  \texttt{37B.8.1},\,  \texttt{37B.8.2};\,\, \texttt{pNn}, $\,\,\,p\ge 17$ \\
\cline{2-3}
 & 1 &\texttt{7Ns.3.1}, \, \texttt{11B.10.4},\,  \texttt{11B.10.5},\,  \texttt{17B.4.2},\,  \texttt{17B.4.6}  \\
\hline
\multirow{2}{*}{$\#X_G(\Q)=\infty$}  & $1$ & \texttt{11Nn} \\
\cline{2-3}
 & $0$ & otherwise \\
\hline
\end{tabular}
\end{center}
The main goal of our project is to characterizes the exceptional and adically-exceptional types of non-CM elliptic curves defined over $\Q$. 

Let $E$ be a non-CM elliptic curve defined over $\Q$. Let $S\subseteq S_E$ be a subset of the set of exceptional primes for $E$ and $[\pm G_E(p)\,:\,p\in S]$ be the exceptional type associated to $E$ and $S$ where we add $-I$ to each component if it is not already there. Let $G\subseteq \GL_2(\Z/n\Z)$ be the group corresponding to the type $[G_E(p)\,:\,p\in S]$, where $n$ is the product of $p\in S$ (i.e. the level of $[G_E(p)\,:\,p\in S]$). The group $G$ satisfies $-I\in G$ since we added $-I$ to each $G_E(p)$ and  $\det(G)=(\Z/n\Z)^\times$ and as discussed before we can associated a modular curve $X_G$. Using the correspondence established in Section \ref{sec:maintheorems}, there exist a non-CM elliptic curve with mod $p$ image in $\pm G_E(p)$ for all $p\in S$ exactly when there exists a non-cuspidal, non-CM rational point in $X_G(\Q)$. The case of adically-exceptional types is equivalent. We are primarily interested in those groups $G$ coming from (adically-)exceptional types such that $\#X_G(\Q)=\infty$. That is, when $X_G$ is a genus 0 curve with a rational point, i.e. $X_G(\Q)\simeq \mathbb P^1(\Q)$, or $X_G$ is an elliptic curve defined over $\Q$ with positive rank over $\Q$. Although this is our primary interest, in this paper we build the basis to finalize the characterization of the (adic-)types that can occur for non-CM elliptic curves over $\Q$, as well as the combinations of primes that can be (adically-)exceptional for a given non-CM elliptic curve $E/\Q$. This project will have a second component \cite{DGJ2}.

Let $G$ be a group in Table \ref{jmaps} such that $\#X_G(\Q)<\infty$. Thanks to Corollary \ref{cor:only_depends_on_j} we know that Serre's constant is an invariant of the $\Qbar$-isomorphism class. For each non-CM $j$-invariant $j_0$, it is enough to take one elliptic curve $E/\Q$ with $j_E=j_0$ and then compute the set $S_E$ of exceptional primes. Zywina \cite{zywina2} gave an algorithm to compute $S_E$ and combined  with his classification in Table \ref{jmaps} allow us to determine the exceptional type of the $\Qbar$-isomorphic class of $E$. We can compute the set of adically-exceptional primes $S_E^\infty$ in an analagous manner. Then we obtain that the possible (adically-)exceptional types are the ones that appear in Table \ref{tab:jconstant}. Assuming Conjecture \ref{13S4} and Strong Uniformity gives that Table \ref{tab:jconstant} is complete.

Let $p$ be prime, after the above sieve, we have $29$ groups $G$ of the form $G_E(p)$ (see Table \ref{jmaps}) and $5$ of the form $G_E(p^k)$ for $k\ge 2$ (Table \ref{tab:adic_groups}) such that $-I\in G$ and $\#X_G(\Q)=\infty$. Moreover, all those curves have genus 0, except for $G=\texttt{11Nn}$, that is an  elliptic curve with positive rank. Of these 34 images that occur for infinitely many $\Qbar$-isomorphism classes, there levels are broken down in the following table. 
\

\begin{center}
\begin{tabular}{|c||c|c|c|c|c|c|c|c|c|}\hline
Level & 2 & 3 & 4 & 5 & 7 & 8 & 9 & 11 & 13\\\hline
\# of groups & 3 & 4 & 2 & 9 & 6 & 2 & 1 & 1 & 6\\\hline
\end{tabular}
\end{center}
\subsection{Pairs of non-surjective Galois representations.}
For the remainder of this section any group $G$ is one of the above $34$ groups. In the next sections we treat separately the cases of exceptional primes and adically-exceptional primes. Note that we are trying to characterize the complete set of possible combination of mod $p$ Galois representations, but in the $p$-adic case we are only interested up to conjugation of a subgroup of a maximal group in $\GL_2(\Z_p)$.

\subsubsection{Exceptional pairs}\label{sub_extype} Let $E/\Q$ be a non-CM elliptic curve with two distinct exceptional primes $p<q$ and let $G_p=G_E(p)$ and $G_q=G_E(q)$. The next sieve comes from the classification of rational $n$-cyclic isogenies given by Mazur and Kenku (cf. \cite{Mazur1978,kenku2, kenku3, kenku4, kenku5}) and torsion structure given by Mazur (cf. \cite{mazur1}). For the special case of  a non-CM elliptic curve $E/\Q$ we have:
\begin{itemize}
\item if $E$ has a cyclic $n$-isogeny, then $n\in\{1,\dots,13,15,16,17,18,21,25,37\}$, and 
\item $E(\Q)_{\tors}\in\left\{\Z/m\Z\,:\, m=1,\dots,10,12\right\}\cup \left\{\Z/2\Z\oplus \Z/2m\Z \,:\, m=1,\dots,4\right\}$.
\end{itemize}
Using the above classifications we can immediately eliminate some pairs $G_1,G_2$. For example, any combinations of groups $G_5\subseteq \texttt{5B}$ and $G_7\subseteq \texttt{7B}$ can be eliminated, since such a elliptic curve would have a 35-isogeny. After using these results to eliminate combinations of groups that would violate these classifications, we are left with 206 possible pairs to consider. For each of these possibilities, we compute the corresponding group $G\subseteq \GL_2(\Z/pq\Z)$, and we compute the genus of the corresponding modular curve $X_G$ group theoretically (see \cite[Lemma 2.4]{SZ}). Note that $X_G=X_{G_p}\times_{\mathbb P^1}X_{G_q}$, the fiber product:
$$
    
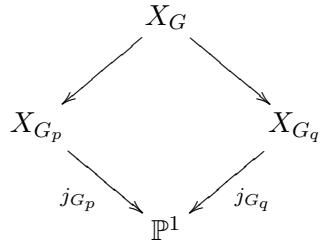
\xymatrix{   &\ar[dl] X_G \ar[dr] &  \\
              X_{G_p}  \ar[dr]_{j_{G_p}} &  & X_{G_q}  \ar[dl]^{j_{G_q}}\\
               & \mathbb P^1 & }
$$
Therefore,
$$
X_G(\Q)=\{(R_p,R_q)\in X_{G_p}(\Q)\times X_{G_q}(\Q)\,:\, j_{G_p}(R_p)=j_{G_q}(R_q) \}.
$$
The simplest case is that the modular curves associated to the groups $G_p$ and $G_q$ both have genus 0. That is $G_p,G_q\ne \texttt{11Nn}$. In this case, both curves $X_{G_p}$ and $X_{G_q}$ are isomorphic to $\PP_1$ and an equation of the fiber product $X_G$ is just the numerator of $j_{G_p}(x) - j_{G_q}(y) = 0$. This model for this curve is usually singular, but give an initial model to work with in order to classify all the rational points. 

The next case we considered is when, say $G_{q}=\texttt{11Nn}$. In \cite[\S 4.5.5]{zywina1}, Zywina gives polynomials $A(x),B(x),C(x)$ in $\Z[x]$ such that an elliptic curve $E$ has $G_E(11)$ conjugate to a subgroup of \texttt{11Nn} exactly when the polynomial $j_E^2A(x) + j_EB(x) + C(x)$ has a rational root, where $j_E$ denotes the $j$-invariant of $E$. Thus, in order to construct a plane curve model for the modular curve that parameterizes elliptic curves over whose image mod $p$ image is conjugate to a subgroup of $G_p$ and whose mod $11$ image is conjugate to a subgroup of \texttt{11Nn} can consider the curve given by the equation $j_{G_p}(y)^2A(x) + j_{G_p}(y)B(x) + C(x)=0$. 

\

The largest genus that occurs in this computation is genus 246, and of the 206 curves 141 of these curves have genus less than 20. The counts of genus curves are listed in the table below.
\

\begin{center}
\begin{tabular}{|c||c|c|c|c|c|c|c|c|c|c|c|c|c|c|c|c|c|c|c|}\hline
$g$ & 0 & 1 & 2 & 3 & 4 & 5 & 6 & 7 & 8 & 9 & 10 & 11 & 13 & 14 & 15 & 16 & 18 & 19 & $\ge$ 20\\\hline
\# Curves of genus $g$ & 12 & 25 & 14 & 15 & 12 & 8 & 5 & 9 & 7 & 7 & 2 & 2 & 4 & 5 & 4 & 3 & 1 & 6 & 65\\\hline
\end{tabular}
\end{center}
\

Here we point out that the genus of the modular curve $X_G$ can computed without ever computing a model for the curve. This is becuase the genus of $X_G$ can be computed by counting elliptic points and cusps on $X_G$ (see \cite[Theorem 3.1.1]{DS} for example) which can be done using just $G$. The code for the computation of the genus of these curves was taken from \cite{SZ}. 

\

\noindent We obtain the following data depending on the genus of the fiber product:\\

{\noindent $\bullet$ Genus 0:}  For each of the genus 0 curves we are able to compute the $j$-maps and thus completely classify all of the elliptic curves with those exceptional types. This data can be found in Table \ref{jmapsj0coarse} and the corresponding parametrization to obtain the fine moduli in Table \ref{jmapsj0fine}. \\
{\noindent $\bullet$ Genus 1:}  For the 25 genus 1 curves (see the first two columns of Table \ref{tab:gen1nonCM}) there are 24 elliptic curves and 1 curve that does not have a single rational point. Of the 24 elliptic curves, 20 have rank 0 and 4 have positive rank. For the curves with rank 0, we compute all of the rational points and their corresponding $j$-invariants to see that only 6 curves have points that correspond to non-CM elliptic curves. These curves, their points and representatives of the corresponding non-CM $\Qbar$-isomorphism classes can be found in Table \ref{tab:gen1r0} and the finitely many elliptic curve that correspond to the fine moduli in Table \ref{tab:gen1r0fine}. Lastly, the modular curves that are genus 1 and have positive rank we give the $j$-maps and the Cremona label for the corresponding model in Table \ref{tab:gen1PosRank} (except for the case \texttt{[3Nn,5S4]}, see Section \ref{proofs}). \\
{\noindent $\bullet$ Genus 2:}  For each of the genus 2 curves the rank of their jacobian have rank 0 or 1, then we can apply Chabauty to obtain all the rational points. We have obtained that all the rational points, if there are, correspond to CM $j$-invariants. This data can be found in Table \ref{tab:g2}.\\
 {\noindent $\bullet$ Genus > 2:} There are 155 curve of genus > 2. For the corresponding groups, there are 28 maximal groups. Now, thanks to the non-constant morphism $X_G\to X_{G'}$ when $G\subseteq G'$ , we have that it is enough to compute the rational points of the curves corresponding to these 28 curves. First we checked for CM-points and then we have looked for points of bounded height. For the cases that $G_i\ne \texttt{11Nn}$, $i=1,2$ with bound $10^6$, otherwise with bound $100$ or until the computer used more than $50$GB of memory. In all those maximal curves we have not found any non-CM rational point. For each type we add a subscript indicating the genus of the corresponding modular curve, a superscript of \text{cm} when the only points we have found correspond to elliptic curves with CM or a superscript of $\emptyset$ in the case that we have not found any rational point. This data can be found in Table \ref{rema}.

\subsubsection{Adically-exceptional pairs} In the case of $p$-adic Galois representations, our objective in this paper is only to characterize the possible combinations of maximal $p$-adic images. That is up to conjugation of a subgroup of a maximal group in $\GL_2(\Z_p)$.

Let $E/\Q$ be an elliptic curve with at least one adically-exceptional prime $p$ such that $p$ is not exceptional. That is $G_E^\infty(p) \ne \GL_2(\Z_p)$ and $G_E(p) = \GL_2(\Z/p\Z)$. Therefore $G_1=G_E^\infty(p)$ is conjugate to a subgroup of one of the groups listed in Table \ref{tab:adic_groups}. 
Now, let $q\ne p$ a prime and $G_2=G^\infty_E(q)$ one of the groups in Table \ref{jmaps} such that the corresponding modular curve has infinitely many points or in Table \ref{tab:adic_groups} such that it is maximal.
All those groups in Table  \ref{tab:adic_groups} are maximal, meanwhile the maximal groups in Table \ref{jmaps} are \texttt{2B}, \texttt{2Cn}, \texttt{3B}, \texttt{3Nn}, \texttt{5B}, \texttt{5Nn}, \texttt{5S4}, \texttt{7B}, \texttt{7Nn}, \texttt{7Ns}, \texttt{11Nn} and \texttt{13B}
\footnote{ One might expect to see \texttt{3Ns} and \texttt{5Ns} on this list of groups, but due to the unique characteristics of 3 and 5 these groups are in fact not maximal. One can check that in these cases, $\texttt{3Ns}\subsetneq \texttt{3Nn}$ and $\texttt{5Ns}\subsetneq \texttt{5S4}$.  }. 
Let $G_1=G^\infty_E(p)$ and $G_2=G^\infty_E(q)$. Similarly to the exceptional case at section \ref{sub_extype}, for each of these possibilities, we compute the corresponding group $G\subseteq \GL_2(\Z/p^k q\Z)$, the genus and a model of the corresponding modular curve $X_G$.  
In this case we obtain 54 curves, and the largest genus is 111. 
Note that for these computations we have not made the cases when $G_1$ and $G_2$ belong to Table \ref{jmaps}, since those computations have been done in the previous section. The counts of genus curves are listed in the table below.
\

\begin{center}
\begin{tabular}{|c||c|c|c|c|c|c|c|c|c|c|c|c|c|c|c|c|c|c|c|}\hline
$g$ & 0 & 1 & 2 & 3 & 4 & 6 & 7 & 8 & 10 & 13 & 14 & 18 & 26 & 40 & 54& 111 \\ 
\hline
\# Curves of genus $g$ & 10 & 15 & 6 & 7 & 2 & 2 & 3 & 1 & 1 & 1 & 1 & 1& 1& 1& 1& 1 \\\hline
\end{tabular}
\end{center}
\

\noindent
We obtain the following data depending on the genus of the fiber product:\\

{\noindent $\bullet$ Genus 0:} Similar to the Section \ref{sub_extype} although in this case we do not consider the fine moduli spaces. The data can be found in Table \ref{jmapsj0coarse}. \\
{\noindent $\bullet$ Genus 1:}  There are 15 genus 1 curves (see the first two columns of Table \ref{tab:gen1nonCM}), only 2 of them without rational points. For the remaining elliptic curves there are 8 with rank 0  and 5 with positive rank, and the corresponding data appear in Table  \ref{tab:gen1r0} and  Table \ref{tab:gen1PosRank} respectively.\\
{\noindent $\bullet$ Genus 2:}  There are 6 genus 2 curves. One of then has not any rational point, 4 have jacobian with rank 0 or 1 and then we obtain all their rational points applying Chabauty. The remaining curve is the modular curve associated to the pair  [\texttt{4X7},\texttt{5Nn}] whose jacobian has rank 2. Then in order to obtain its rational points we apply $2$-cover descent and elliptic Chabauty. This computation has been realized with the help of Xavier Xarles. Let $C$ be the following hyperelliptic model associated to the modular curve corresponding to the pair [\texttt{4X7} \texttt{5Nn}]:
$$
C\,:\,y^2=f(x)\,,\qquad f(x)=(x-2) \left(x^2+1\right) \left(4 x^3-4 x^2+3 x-2\right).
$$
First we do a point search to find all easily visible points and we compute that in projective coordinates
$$
S_1:=\{[0 : \pm 2:1], [3/4;\pm 50/4:1], [2 : 0 : 1],[1:\pm 2:0]\}\subseteq C(\Q).
$$
Our purpose is to prove $C(\Q)=S_1$. Let us determine an unramified two covering $\psi :D\to C$ defined over $\Q$ such that the associated covering collection $\psi_\delta:D_\delta\to C$ satisfies that any rational point $P\in C(\Q)$ on the curve lifts to one of the covers $D_\delta(\Q)$. The idea is to factorize the polynomial
$f(x)$ as the product of two polynomials (over some number field $K$) of even degree. Let be $f(x)=f_1(x)f_2(x)$, where $f_1(x),f_2(x)\in K[x]$ for some number field $K$. We get the subcovers $E_1:y_1^2=f_1(x)$ and $E_2:y_2^2=f_2(x)$ and the unramified two covering $\psi :D\to C$,  $\psi(x_0,y_1,y_2)=(x_0,y_1y_2)$ where $D\,:\,\{y_1^2=f_1(x)\,,\,y_2^2=f_2(x)\}$:
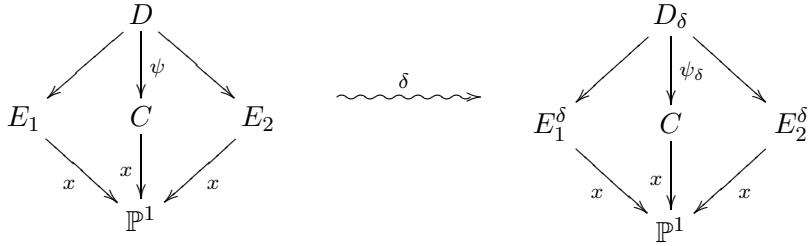
$$
\begin{array}{ccc}
\xymatrix@R=2pc@C=0.8pc{
& & D \ar@{->}[d]^{\psi}    \ar@{->}[drr]  \ar@{->}[dll]   &   &      \\
E_1 \ar@{->}[drr]_{x} & & C     \ar@{->}[d]_{x}												 & &E_2 \ar@{->}[dll]^{x}\\
      				 & &	 \mathbb{P}^1  & & 			
}
& 
\xymatrix@R=2pc@C=0.8pc{
 & & & &\\
  \ar@{~>}[rrrr]^\delta & & & &\\
 & & & &
}
&
\xymatrix@R=2pc@C=0.8pc{
& & D_\delta \ar@{->}[d]^{\psi_\delta}    \ar@{->}[drr]  \ar@{->}[dll]   &   &      \\
E_1^\delta \ar@{->}[drr]_{x} & & C     \ar@{->}[d]_{x}												 & &E_2^\delta \ar@{->}[dll]^{x}\\
      				 & &	 \mathbb{P}^1  & & 			
}
\end{array}
$$
Now if we can determine $D_\delta(\Q)$ we determine $C(\Q)$, since $D_\delta(\Q)$ maps to $\{P\in E_i^\delta(K):x(P)\in \mathbb P^1(\Q)\}$. Finally, in order to determine which of those points correspond to points in $C(\Q)$ we only have to determine wich points in $x(E_i^\delta(K))\cap \mathbb P^1(\Q)$ lift to $C(\Q)$, for any of the two subcovers. \\
\indent Let be $K=\Q(\alpha)$ where $\alpha$ is a root of the polynomial $g(x)=4 x^3-4 x^2+3 x-2$ then we choose the following factorization of the polynomial $f(x)$:
$$
f_1(x)=(x - 2)(x-\alpha),\qquad\qquad 
f_2(x)=(x^2 + 1)(4x^2 + 4(\alpha - 1)x + (4\alpha ^2 - 4\alpha + 3).
$$
We could have chosen other factorizations over other number fields, but this one works for our purpose. Now, in order to compute the (finite) set $\mathcal T$ of twists necessary to cover all the rational points we compute the Fake 2-Selmer group of $C/\Q$ (see \cite{BS}). This can be done by the \texttt{Magma} function \texttt{TwoCoverDescent}. We check that all the possible twists come from the points in the set $S_1$: let $\delta\in\mathcal T$, then $\delta=f_2(x_0)\in K^*/K^{*2}$ for $x_0$ the $x$-coordinate of an affine point in $S_1$; and $\delta=1$. Note that $f_2(3/4)\in K^2$. Therefore we have obtained $\mathcal T=\{1\}\cup\{f_2(x_0):x_0\in\{0,2\}\}$. For any $\delta\in\mathcal T$ we have that $\text{rank}_\Z(E_2^\delta(K))< 3$,  then we can apply Elliptic Chabauty to the covering $x:E_2^\delta\to\mathbb P^1$, to obtain the set $C(\Q)$. The following tables illustrates the data obtained for each $\delta$ the corresponding point in $\mathbb P^1(\Q)$:
$$
\begin{array}{|c|c|c|c|}
\hline
\delta & 1 & \alpha^2 - \alpha + 3/4 & 5\alpha^2 + 5\alpha + 55/4 \\
\hline
x(P)\in \mathbb P^1(\Q), P\in E_2^\delta(K) & [3/4:1], [1:0] & [0:1] & [2:1] \\
\hline 
\end{array}
$$
Therefore we finish with $C(\Q)=S_1$.
\\
{\noindent $\bullet$ Genus > 2:}   Similar to the case in section \ref{sub_extype}: the data can be found in Table \ref{rema}. But in this case we have found some non-CM rational points:
\begin{itemize}
\item $j=3^3 5\, 7^5/2^7$ in the modular curve of genus 3 asociated to \texttt{[8X5,7Ns]}. But this $j$-invariant corresponds to \texttt{[8X5,7Ns.3.1]} (see Table \ref{tab:jconstant}).
\item  $j=-2^2 3^7 5^3 439^3$ in the modular genus 6  curve associated to the pair  \texttt{[4X7,9XE]}. 
\end{itemize}

\section{ Proof of the Theorem \ref{the_comb_modp}, Theorem \ref{the_comb_adic}, Corollary \ref{the_comb_modp_COR} , and  Corollary \ref{the_comb_adic_COR}}\label{proofs}

We have discussed previously that our computations allow us to compute the genus of the modular curves obtained as the fiber product of modular curves arising from pairs of non-surjective (mod $p$ or $p$-adic) Galois representations. In particular, such a curves can have infinitely many non-CM points and non-cusps if and only if it has genus 0 with some rational point, or it is an elliptic curve with positive rank.  Apart from the case with a single group (see Tables \ref{jmaps}, \ref{table_mod_single},  \ref{tab:adic_groups}), we have obtained the following:\\
{$\bullet $ Genus 0:} For each curve of genus zero we check that it has a rational point. See Tables \ref{jmapsj0coarse} and  \ref{jmapsj0fine} for all those possible pairs.\\
{$\bullet $ Genus 1:} There are $40$ modular curves of genus 1 associated to pairs of mod $p$ and $p$-adic non-surjective representations. From them, the $9$ elliptic curves of positive rank appear in Table \ref{tab:gen1nonCM}. Note that the cases \texttt{[8X5,3Nn]} and \texttt{[3Nn,5S4]} do not appear at Theorems \ref{the_comb_modp}  and  \ref{the_comb_adic}. In the following paragraphs we describe the reasons that we can discard those cases:
\begin{itemize}
\item[$\star$] \texttt{[8X5,3Nn]}: Let $\mathcal E$ be the modular curve associated to \texttt{[8X5,3Nn]}. In this case, $\mathcal E/\Q$ is the elliptic curve with Cremona label \texttt{576a3} and has $j$-map equal to $j(x)=8x^3$ where $(x,y)\in \mathcal E(\Q)$. Now let \texttt{8X17} be the group
$$
\left \langle \smallmat{1}{0}{2}{1},\smallmat{1}{0}{0}{7},\smallmat{1}{1}{0}{5}\right\rangle \subsetneq \texttt{8X5} \subsetneq \GL_2(\Z/8\Z).
$$
The modular curve associated to the (nonmaximal) exceptional pair [\texttt{8X17},\texttt{3Nn}] is the elliptic curve $\mathcal E'/\Q$ with Cremona label \texttt{576a1}. The $j$-map $\mathcal E'$ is equal to $j'(x)=8\left(\frac{x^3+32}{x^2}\right)^3$ where $(x,y)\in \mathcal E'(\Q)$. As the Cremona labels indicate, there is an isogeny $\varphi:\mathcal E'\rightarrow \mathcal E$ defined by 
$$
\varphi(x,y)=\left(\frac{x^3+32}{x^2},\frac{x^3y-64y}{x^3}\right),
$$
of degree 3. 
A simple calculation shows that $j'(x,y)=j(\varphi(x,y))$ and thus ${\rm Im}(j)={\rm Im}(j')$. Because of the relationship between these two groups and the $j$-maps of these modular curves, any elliptic curve that have potentially $G^\infty_E(2)$ conjugate to a subgroup of \texttt{8X5} and $G^\infty_E(3)$ conjugate to a subgroup of \texttt{3Nn} must have (smaller) images: $G^\infty_E(2)$ conjugate to \texttt{8X17}. Further, \texttt{8X17} is a subgroup of \texttt{2B} and so the mod 2 Galois representation was already non-surjective. Thus, there are no elliptic curves of { exact} type \texttt{[8X5,3Nn]} despite the fact that $\mathcal{E}(\Q)$ contains infinitely many points. 
\item[$\star$] \texttt{[3Nn,5S4]}: Similarly the { exact} type \texttt{[3Nn,5S4]} does not actually occur for any elliptic curves over $\Q$ despite the fact that the modular curve corresponding to this type has infinitely many rational points. This is again because the modular curves for \texttt{[3Nn,5S4]} and \texttt{[3Nn,5Ns]} are isogenous elliptic curves and their $j$-maps are related in the same way as the curves above. The group \texttt{5Ns} is a proper subgroup of \texttt{5S4} and thus every elliptic curve coming from a rational point on the modular curve for \texttt{[3Nn,5S4]} also comes from a point on the modular curve for the smaller type \texttt{[3Nn,5Ns]}. In this case the modular curve associated to \texttt{[3Nn,5S4]} has Cremona label \texttt{225a2} and for \texttt{[3Nn,5Ns]} is \texttt{225a1}. These elliptic curves both have rank 1 and are 3-isogenous to each other. 
\end{itemize}
To summarize the above two cases, any elliptic curve of type \texttt{[8X5,3Nn]} (resp.~\texttt{[3Nn,5S4]}) must also be of smaller type \texttt{[8X17,3Nn]} (resp.~\texttt{[3Nn,5Ns]}). So there are no elliptic curves of exact adic type \texttt{[8X5,3Nn]} (resp. exact type ~\texttt{[3Nn,5S4]}) as all the curves of these types much have strictly smaller exact (adic-)type. 

{ We also point out that we know that there are infinitely many elliptic curves whose image is conjugate to the groups listed in Table \ref{table_mod_single} since \cite{zywina1} give explicit 1-parameter families of elliptic curves with images exactly equal to each group outside of a thin set.}

We are now ready to prove Corollary \ref{the_comb_modp_COR}

\begin{proof}[Proof of Corollary \ref{the_comb_modp_COR}]
The proof of this corollary breaks down into two cases. Given a group $G\subseteq \GL_2(\Z/n\Z)$ with $n$ squarefree satisfying $-I\in G$ and $\det(G)=(\Z/n\Z)^\times${, from Faltings' Theorem \cite{Faltings}} there can only be infinitely many $\Qbar$-isomorphism classes of elliptic curves over $\Q$ of type $G$ if the corresponding modular curve $X_G$ is genus 0 with a rational point or if it is an elliptic curve with positive rank over $\Q$. 

For the first case, assume $X_G\simeq \mathbb P^1$ over $\Q$. Under the assumption of uniformity we may assume that $p$ divides $n$ if and only if there exists some elliptic curve over $\Q$ for which $p$ is exceptional. This is because if $p$ did not already divide $n$, we can lift the group $G$ to be a subgroup of $\GL_2(\Z/pn\Z)$ by taking its preimage under the standard reduction map. This processes does not affect the modular curve $X_G$ or its moduli interpertation.

By \cite[Lemma 3.5]{zywina1} we have that there are infinitely many $\Qbar$-isomorphism classes of elliptic curves with $\pm G_E(n)$ conjugate to $G$. For each of these $\Qbar$-isomorphism classes, the curves in them cannot have any additional exceptional primes by the assumption that any prime that could potentially be exceptional (under the assumption of uniformity) is already accounted for since it divides $n$. { Thus, for each of the infinitely many $\Qbar$-isomorphism classes of elliptic curves guaranteed to exist by \cite[Lemma 3.5]{zywina1}, Remark \ref{rmk:twisting_and_-I} ensures that there is at least one $\Q$-isomorphism class with image exactly $G$.}

%
%

Finally, assume that the corresponding modular curve $X_G$ has genus 1 and positive rank. For every $p$ dividing $n$ let $G_p = \pi_p(G)$ where $\pi_p\colon \GL_2(\Z/n\Z) \to \GL_2(\Z/p\Z)$ is the standard componentwise reduction map. Then every point in $X_G(\Q$) corresponds to a $\Qbar$-isomorphism class of elliptic curves with type $[G_p : \hbox{$p$ dividing $n$}]$. 

Suppose that $E$ has type $[G_p : \hbox{$p$ dividing $n$}]$ but its exact type is not $[G_p : \hbox{$p$ dividing $n$}]$. Then at least one of the following must be true: 
\begin{itemize}
\item The level of the type of $E$ is larger than $n$.
\item There is a prime $p$ that divides $n$ such that $G_E(p)$ is conjugate to a proper subgroup of $G_p$. 
\end{itemize}

{ From the assumption of uniformity we know that there are only finitely many primes $p$ such that $p\nmid n$ and there are elliptic curves defined over $\Q$ that are exceptional at $p$. Of these finitely many primes $p$, there are only finitely many ways that an elliptic curve $E/\Q$ can be of the type associated to $G$ and exceptional at $p$. The modular curves corresponding to each of these possibilities are shown to all have have genus greater than 1 or to be genus 1 with rank 0 in Section \ref{outline}. Therefore, there can only be finitely many $\Qbar$-isomorphism classes of elliptic curve of type $G$, but exact type of level larger than that of $G$.}

Next, checking every possible subtype shows that the are only 2 types whose corresponding modular curves are of genus 1 and have positive rank over $\Q$ that also have a subtype with a genus 1 positive rank corresponding modular curve. Those types are \texttt{[8X5,3Nn]} and \texttt{[3Nn,5S4]} and we dealt with these at the beginning of Section \ref{proofs}.

%
\end{proof}

{Since we do not require that the types be exact in Corollary \ref{the_comb_adic_COR}, the result follows from the genus and rank computation that we have described in this section as well as an argument similar to that of the proof of Corollary \ref{the_comb_modp_COR}.}

\

Now for a given (adic-)type to occur for infinitely many $\Qbar$-isomorphism classes of elliptic curves it must be that the corresponding moduli space is genus 0 with a rational point or an elliptic curve with positive rank over $\Q$. { This is sufficient to conclude that there are infinitely many $\Qbar$-isomorphism classes of elliptic curves of \emph{exact} (adic)-type thanks once again to \cite[Lemma 3.5]{zywina1}.} Calculating all the possibilities gives the lists that appear in part (A) of Theorems \ref{the_comb_modp}  and  \ref{the_comb_adic} and the proof of Corollary \ref{the_comb_modp_COR} shows that the list in Theorem \ref{the_comb_modp} part (A) is in fact complete. {Further, a similar analysis of the computations completes the proof of Theorem  \ref{the_comb_adic}.}

In fact, with a further computation, we can prove the following lemma as well.

\begin{lemma}\label{3primes}
There are no (adically-)exceptional types of length 3 such that there exist infinitely many non-isomorphic classes of non-CM elliptic curves over $\Q$.
\end{lemma}

\begin{proof}
Let us suppose that there is a type of length 3 of the form $[G_1,G_2,G_3]$ such that there exist infinitely many non-isomorphic classes of non-CM elliptic curves over $\Q$ with that type. In order for this to be the case it must be that all the subtypes of length 2, $[G_i,G_j]$, occur for infinitely many $\Qbar$-isomorphism classes of elliptic curves. Looking for all the possibilities in all the pairs at (A) in Theorems \ref{the_comb_modp} and \ref{the_comb_adic} we have that the unique possibility is \texttt{[3Nn,5B,2B]}. The modular curve associated to \texttt{[3Nn,5B,2B]} is equivalent to the fiber product of $X_\texttt{3Nn}$ and $X_\texttt{10B}=X_0(10)$. 
This modular curve\footnote{A remarkable fact is that this genus 2 curve is new modular of level 90 and its jacobian is $\Q$-isogenous to the product of two elliptic curves (see \cite{baker}).} has genus 2 and an hyperelliptic model is $C:y^2 = x^6 - 18x^3 + 1$. Since its jacobian has rank 0  we can apply Chabauty to obtain all the rational points. We have obtained that all the rational points correspond to cusps. Therefore there not exist any triples $[G_1,G_2,G_3]$ of exceptional types or adically-exceptional types such that there exist infinitely many non-isomorphic classes of non-CM elliptic curves over $\Q$ of that type.
\end{proof}

{
In order to complete the proof of Theorem \ref{thm_serre_constant} we still need to justify that there are infinitely many distinct $\Qbar$-isomorphism classes of elliptic curves defined over $\Q$ with Serre's constant corresponding to the levels of all the modular curves associated to the types that appear in Tables  \ref{jmapsj0coarse}, \ref{jmapsj0fine} and \ref{tab:gen1PosRank}. We have justified already that all the others appears, but there is a subtlety that still needs to be sorted out with 104.}  While we have seen that there are infinitely many $\Qbar$-isomorphism classes of elliptic curves of type \texttt{[8X4,13B]} since the modular curve corresponding to \texttt{[8X4,13B]} is genus 1 of positive rank. It is possible that for all but finitely many of these $\Qbar$-isomorphism classes the mod 2 or mod 4 image is not actually surjective while still being contained inside \texttt{8X4}. Therefore, all but finitely many of these curve might have smaller Serre's constant (either 26 or 52). 
We start by showing that if $E/\Q$ has type \texttt{[8X4,13B]}, then the mod 2 Galois representation associated to $E$ must be surjective. 
Since having mod 8 image contained in the group associated to \texttt{8X4} corresponds to having $\Delta(E) = -2t^2$ for some $t\in \Q$ we know that the only way that $E$ could have non-surjective mod 2 image would be for $E$ to have a rational point of order 2. This, of course, is impossible for an elliptic curve of type \texttt{[8X4,13B]} since any such curve already has a 13-isogeny and there are no elliptic curves over $\Q$ with a 26-isogeny. 
Lastly, using the database in \cite{RZB} we see that in order for an elliptic curve to have surjective image mod 2, be of type \texttt{[8X4,13B]}, and to not have surjective image mod 4 is for the mod 4 image to be contained in the group associated to \texttt{4X7}. This would mean that $E$ also has type \texttt{[4X7,13B]}, but our computations show that the modular curve corresponding to \texttt{[4X7,13B]} is genus 3 and so there can be at most finitely many such $\Qbar$-isomorphism classes. Thus, there must be infinitely many curves with Serre's constant $104$ and we have completed the proof of Theorem \ref{thm_serre_constant}. That is, $\mathcal A_\infty$ is the set consisting of the levels of all the modular curves associated to the types that appear in Tables  \ref{jmapsj0coarse}, \ref{jmapsj0fine} and \ref{tab:gen1PosRank}. 

\

{To finish our project} we still need to complete the classification of possible (adic-)types and Serre's constant. To do this we fix our attention to the associated modular curves with finitely many rational points. That is, genus 1 and rank 0, or genus >1:\\
{$\bullet $ Genus 1 and rank 0:}  There are $28$ elliptic curves of rank 0. Only $8$ have points that correspond to non-CM elliptic curves. All the data appears in Tables \ref{tab:gen1r0} and \ref{tab:gen1r0fine}. {The case \texttt{[4X7,3B]} does not appear in that table: 
\begin{itemize}
\item[$\star$] The modular curve associated to the type \texttt{[4X7,3B]} is the elliptic curve $\mathcal E/\Q$ with Cremona label \texttt{48a6} and with Mordell-Weil group $\mathcal E(\Q)\simeq \Z/8\Z$. These points give the non-CM $j$-invariants $j_1=-3^3\cdot 11^3/2^2$ and $j_2=3^2\cdot 23^3/2^6$. Now, let $j(t)$ be the j-map (see Table \ref{jmapsj0coarse}) of the genus 0 modular curve associated to the type \texttt{[4X3,3B]}. Therefore $j(-1/2)=j_1$ and $j(-1)=j_2$. Therefore there are no elliptic curves with type \texttt{[4X7,3B]} as any curve with this combination of images actually has smaller images \texttt{[4X20,3B]}.  Here \texttt{4X7} referes to the level 4 group \texttt{X20} in the notation of \cite{RZB}.
\end{itemize}
{$\bullet $ Genus = 2:}  We have computed all the rational points of all the modular curves of genus 2. In all the cases we have not obtained non-CM elliptic curves, except in the case \texttt{[4X7,7B]}: 
\begin{itemize}
\item[$\star$] The associated modular curve to the type \texttt{[4X7,7B]} has genus 2 and we have computed all its rational points. In this cases the non-CM non-singular j-invariants are $j_1=3^3\,13/2^2$ and $j_2=-3^3\cdot 13\cdot 479^3/2^{14}$ (see Table \ref{tab:g2}). That corresponds to evaluating the $j$-map of the genus 0 modular curve associated to the type \texttt{[4X3,7B]} at the values $7/2$ and $2$ respectively.  Therefore any elliptic curve $E$ with those $j$-invariants satisfies that $G^\infty_E(2)$ is conjugate to a subgroup of \texttt{4X3} and \texttt{4X7}, and $G^\infty_E(7)$ is conjugate to a subgroup of \texttt{7B}.
\end{itemize}
{$\bullet $ Genus > 2:}  We did a point search to find all easily visible points of bounded heigh. We have found only non-CM, non-cusps in the modular curve associated to the types \texttt{[8X5,7Ns]} and \texttt{[4X7,9XE]}. But the case \texttt{[8X5,7Ns]} does not appear at the statement:
\begin{itemize}
\item[$\star$] The corresponding modular curve to the type \texttt{[8X5,7Ns]} has genus 3 and searching for points of height less than or equal to $10^6$ yields only one non-CM non-singular $j$-invariant, $j= 3^3\cdot 5\cdot 7^5/2^7$ (see Table \ref{rema}) that corresponds to the unique $j$-invariant for the case $G_E(7)$ labeled as \texttt{7Ns.3.1} (see Table \ref{jmaps}).
\end{itemize}

The above proves (B) of Theorems \ref{the_comb_modp} and \ref{the_comb_adic}.  

\begin{example}
Let $E$ be the elliptic curve given by Weierstrass equation $y^2+xy+y=x^3-126x-552$. This curve has Cremona reference \href{http://www.lmfdb.org/EllipticCurve/Q/50a1}{\texttt{50a1}}, it does not have CM and according to LMFDB \cite{lmfdb}, $G_E(3) =\texttt{3B.1.2}$, $G_E(5) =\texttt{5B.1.3}$, and $G_E(p) = \GL_2(\Z/p\Z)$ for every other prime $p$. Checking the 2-adic representation in \cite{lmfdb} (which uses the data collected in \cite{RZB}) we see that the image of $\rho_{E,2^\infty}$ is the pullback of the group \texttt{8X4}
under the standard reduction map $\pi:\GL_2(\Z_2) \to\GL_2(\Z/8\Z)$. So the 2-adic representation associated to $E$ is not surjective and $A(E) = 120$. Here we point out that if we let $\pi_k:\GL_2(\Z/8\Z) \to \GL_2(\Z/2^k\Z)$ be the standard componentwise reduction map for $k=1,$ and 2, then $\pi_1(H) = \GL_2(\Z/2\Z)$ and $\pi_2(H) = \GL_2(\Z/4\Z)$ where $H = G_E^\infty(2)$. Therefore, ${\rho}_{E,2}$ and ${\rho}_{E,4}$ are both surjective. The modular curve whose points correspond to elliptic curves with these mod 3, 5, and 8 images is genus greater than 1, so there are only finitely many $\Qbar$-isomorphism classes of elliptic curves with this particular combination of images. Moreover, since we have obtained all the rational points at the modular curve associated to \texttt{[3B,5B]} we are done. 
\end{example}

\begin{remark}\label{triples}
In fact we can check that for the four  $j$-invariants coming from the modular curve \texttt{[3B,5B]} have that the image of $\rho_{E,2^\infty}$ is the pullback of the group \texttt{8X4}. Similarly with \texttt{[3B,7B]} and \texttt{8X4}; and \texttt{[3B,5S4]} with \texttt{8X3}. In particular this completes the computation of all the points on the modular curves associated to 
\texttt{[4X3,3B,5S4]}, \texttt{[8X4,3B,5B]} and \texttt{[8X4,3B,7B]}.
\end{remark}

Therefore we have obtained all the possible $j$-invariants for the item (B) at Theorems \ref{the_comb_modp}  and  \ref{the_comb_adic}, except (maybe) the type \texttt{[4X7,9XE]}. The modular curve associated to the former case has genus 6. In this article we have not tried to compute all the rational points of such a curve. This will be done in an ongoing project \cite{DGJ2}.

%
%

\section*{Appendix: Tables}

In this section, we give tables of data that summarize the results that we used in our computations as well as the data that we collected. Tables \ref{jmaps}, \ref{table_mod_single}, and \ref{Eq_finemoduli} are taken from the results in \cite{zywina1} where Zywina does a search for all possible images of the mod $p$ representations associated to non-CM elliptic curves over $\Q$ and then computes the moduli spaces for the ones that actually occur. Throughout \cite{zywina1} Zywina is careful to distinguish between subgroup of $\GL_2(\Z/p\Z)$ that do and do not contain $-I$. This is because if $E$ and $E'$ are quadratic twists of each other then we do not necessarily know that $G_E(p) = G_{E'}(p)$. In this case all that can be said is that $\langle G_E(p), -I\rangle = \langle G_{E'}(p), -I \rangle$. Therefore, if $G\subseteq \GL_2(\Z/p\Z)$ contains $-I$, then $G_E(p)\subseteq G$ if and only if $G_{E'}(p)\subseteq G$, while if $-I\not\in G$ then it is possible that $G_E(p) \subseteq G$ and $G_{E'}(p)\not\subseteq G$. Combining this with the fact that two non-CM elliptic curves $E/\Q$ and $E'/\Q$ are $\Qbar$-isomorphic if and only if they are quadratic twists of each other, we have that the moduli spaces associated to groups containing $-I$ are completely determined by a $j$-map and are coarse moduli spaces (See Table \ref{jmaps}). On the other hand, the moduli spaces associated to groups that do not contain $-I$ are called fine moduli spaces and are given by elliptic surfaces such that each nonsingular specialization is a representative of a $\Q$-isomorphism class with the given type (See Table \ref{Eq_finemoduli}). We have included these three tables for the sake of completeness.

There are places in the tables where the polynomials that need to be written are too complicated to fit in the space provided. In those cases we simplify the entry by defining some notation for factors of the polynomials. The extra data is then presented at the end of the Appendix organized by table number. \\

Below we give a description of each of the tables:
\begin{itemize}
\item Table \ref{jmaps}: For each possible groups $G_E(p) \neq \GL_2(\Z/p\Z)$ with $-I\in G_E(p)$ we give Sutherland's and Zywina's labels, the level, generators, the (possibly constant) $j$-maps, and the Cremona label of an elliptic curve with minimal conductor with mod $p$ image equal to the given group.
\item Table \ref{table_mod_single}: For each possible groups $G_E(p) \neq \GL_2(\Z/p\Z)$ with $-I\not\in G_E(p)$ we give Sutherland's and Zywina's labels, the level, generators, and the Cremona label of an elliptic curve with minimal conductor with mod $p$ image equal to the given group.
\item Table \ref{Eq_finemoduli}: For each possible groups $G_E(p) \neq \GL_2(\Z/p\Z)$ with $-I\not\in G_E(p)$ we give an elliptic curve model for the fine moduli.
\item Table \ref {tab:adic_groups}: For each group $G_E(p^k)$ with $k \geq 2$ that surjects onto $\GL_2(\Z/p\Z)$ we give the level, generators $j$-map, and the Cremona label of an elliptic curve with minimal conductor with mod $p$ image equal to the given group.
\item Table \ref {tab:jconstant}: For each group in Table \ref{jmaps} that only has finitely many $\Qbar$-isomorphism classes, we check if the corresponding $j$-invariants are in the image of the $j$-map of other modular curves. The results of this checking are compiled in this table.
\item Table \ref{jmapsj0coarse}: For each exceptional type $[G_E(p),G_E(q)]$ such that $G_E(p)$ and $G_E(q)$ contain $-I$ and the corresponding modular curve has genus 0 we give the $j$-map associated to the curve and the Cremona label of an elliptic curve with minimal conductor and that type. We do the same thing for  maximal adically exceptional types.
\item Table \ref{jmapsj0fine}: For each exceptional type $[G_E(p),G_E(q)]$ such that $G_E(p)$ and $G_E(q)$ with $-I$ not in one of the groups such that the corresponding modular curve has genus 0 we give the parametrization to obtain a fine moduli  associated to the curve and the Cremona label of an elliptic curve with minimal conductor and that type.
\item Table \ref{tab:gen1nonCM}: For each exceptional type such that the corresponding modular curve is genus 1 we give the the Cremona reference of the modular curve as well as the structure of the Mordell-Weil group of the modular curve in the case of the curve is elliptic, otherwise we show a reason why does not have rational points. We do the same thing for maximal adically-exceptional types.
\item Table \ref{tab:gen1PosRank}: For each modular curve in Table \ref{tab:gen1nonCM} with positive rank, except \texttt{[8X5,3Nn]} and \texttt{[3Nn,5S4]} (see \S \ref{proofs}), we give the $j$-map, Cremona reference of the modular curve, and the Cremona reference of an elliptic curve with minimal conductor (if this is less than $400.000$) and that (adic-)type. 
\item Table \ref{ex3Nn}: For each modular curve in Table \ref{tab:gen1PosRank} with positive rank such that the examples of minimal conductor are greater than 400000 and thus does not have a page in the LMFDB, we give a minimal model of the an elliptic curve with that (adic-)type and its conductor. 
\item Table \ref{tab:gen1r0}: For each modular curve in Table \ref{tab:gen1nonCM} with rank 0 and $-I$ in both groups we give a complete list of the $j$-invariants that correspond to all the $\Qbar$-isomorphic classes with that (adic-)type, and the Cremona labels of examples of elliptic curves with that combination of representations of minimal conductor. 
\item Table \ref{tab:gen1r0fine}: For each modular curve in Table \ref{tab:gen1nonCM} with rank 0 and $-I$ not in both groups we give a complete list of the Cremona labels of all elliptic curves with that (adic-)type.
\item Table \ref{tab:g2}: For each possible (adic-)type whose curve has genus 2 we give a complete list of non-cusps and rational points and their corresponding elliptic curves $j$-invariants.
\item Table \ref{rema}: We give a list of the modular curves that one would classify all the non-cusps and rational points in order to prove Conjecture \ref{conj:A} {under the assumption of uniformity}. For each type we add a subscript indicating the genus of the corresponding modular curve, a superscript of \text{cm} when the only points we have found correspond to elliptic curves with CM or a superscript of $\emptyset$ in the case that we have not found any rational point. {For all those curves we have found only CM and/or cusps, or nothing in the modular curve associated to those types except for the pair \texttt{[4X7,9XE]}, where we have found the $j$-invariant $j=-2^2 3^7 5^3 439^3$. For this $j$-invariant we have checked that does not appear neither for types in Table \ref{jmaps} nor in Table \ref{tab:adic_groups}. Assuming that for the remaining curves those are all the points we are done to prove Conjecture \ref{conj:A} under the assumption of uniformity, since the set of types that appear at Table \ref{rema} corresponds to a set of maximal groups that cover all the possible (adically-)exceptional pairs. Note that if $G'\subseteq G$ then there exists a non-constant morphism $X_{G'}\rightarrow X_G$ and if we prove that $X_G(\Q)$ only corresponds to cusps and CM $j$-invariants, then $X_{G'}(\Q)$ corresponds to cusps and CM $j$-invariants.}

\item Table \ref{adicex}: We give examples of elliptic curves with a given (adic-)type whose modular curves have genus greater than one and whose points correspond to elliptic curves that do not appear in the LMFDB. 
\end{itemize}

\begin{remark}
All the \texttt{Magma} \cite{magma} code to compute the tables in this appendix is available in the online supplement \cite{MagmaCode}. Some of the code for this paper was taken from \cite{Q(3)}.
\end{remark}


{
\extrarowheight = -0.5ex
\renewcommand{\arraystretch}{2.15}
\begin{table}
\begin{footnotesize}
\begin{tabular}{|c|c|c|c|l|c|}
\hline
Sutherland &   Zywina & Level & Generators & j-map & Example\\
\hline
\texttt{2Cs} & $G_1$ & 2 & & $256 \frac{(t^2+t+1)^3}{t^2(t+1)^2}$& \href{http://www.lmfdb.org/EllipticCurve/Q/15a1}{\texttt{15a1}} \\
\hline
 \texttt{2B} &$G_2$ & 2& $\smallmat{1}{1}{0}{1}$ &  $256\frac{(t+1)^3}{t}$ & \href{http://www.lmfdb.org/EllipticCurve/Q/15a4}{\texttt{15a4}}\\
\hline
\texttt{2Cn} &  $G_3$& 2 & $\smallmat{0}{1}{1}{1}$ & $t^2+1728$& \href{http://www.lmfdb.org/EllipticCurve/Q/392b1}{\texttt{392b1}}\\
\hline
\hline
\texttt{3Cs} & $G_1$ & 3 & $\smallmat{2}{0}{0}{2},\smallmat{1}{0}{0}{2}$ & $ 27\frac{(t+1)^3(t+3)^3(t^2+3)^3}{t^3(t^2+3t+3)^3}$& \href{http://www.lmfdb.org/EllipticCurve/Q/175b2}{\texttt{175b2}}   \\
\hline
\texttt{3Ns} &$G_2$ & 3  & $\smallmat{2}{0}{0}{2},\smallmat{0}{2}{1}{0},\smallmat{1}{0}{0}{2}$ & $ 27\frac{(t+1)^3(t-3)^3}{t^3}$& \href{http://www.lmfdb.org/EllipticCurve/Q/1210d1}{\texttt{1210d1}} \\
\hline
\texttt{3B} & $G_3$ & 3  & $\smallmat{2}{0}{0}{2},\smallmat{1}{0}{0}{2},\smallmat{1}{1}{0}{1}$ & $ 27\frac{(t+1)(t+9)^3}{t^3}$ & \href{http://www.lmfdb.org/EllipticCurve/Q/175b1}{\texttt{175b1}} \\
\hline
\texttt{3Nn} & $G_4$ & 3  &$\smallmat{1}{0}{0}{2},\smallmat{2}{1}{2}{2}$ & $t^3$ & \href{http://www.lmfdb.org/EllipticCurve/Q/245a1}{\texttt{245a1}}  \\
\hline
\hline
\texttt{5Cs.4.1} & $G_1$ & 5  & $\smallmat{4}{0}{0}{4},\smallmat{1}{0}{0}{2}$ &  $ \frac{(t^{20}+228t^{15}+494t^{10}-228t^5+1)^3}{t^5(t^{10}-11t^5-1)^5}$& \href{http://www.lmfdb.org/EllipticCurve/Q/99d2}{\texttt{99d2}}\\
\hline
\texttt{5Cs} &  $G_2$ & 5   & $\smallmat{2}{0}{0}{3},\smallmat{1}{0}{0}{2}$ & $\frac{(t^2 + 5t + 5)^3(t^4 + 5t^2 + 25)^3(t^4 + 5t^3 + 20t^2 + 25t + 25)^3}{t^5(t^4 + 5t^3 + 15t^2 + 25t + 25)^5}$  & \href{http://www.lmfdb.org/EllipticCurve/Q/18176b2}{\texttt{18176b2}} \\
\hline
\texttt{5Ns.2.1} & $G_3$ & 5   & $\smallmat{2}{0}{0}{3},\smallmat{0}{1}{3}{0}$ & $\frac{5^4 t^3 (t^2+5t+10)^3 (2t^2+5t+5)^3 (4t^4+30t^3+95t^2+150t+100)^3}{(t^2+5t+5)^5(t^4+5t^3+15t^2+25t+25)^5}$ & \href{http://www.lmfdb.org/EllipticCurve/Q/6975a1}{\texttt{6975a1}}\\
\hline
\texttt{5Ns} & $G_{4}$ & 5   & $\smallmat{0}{4}{1}{0},\smallmat{2}{0}{0}{3},\smallmat{1}{0}{0}{2}$ & $ \frac{(t + 5)^3 (t^2 - 5)^3(t^2 + 5t + 10)^3}{(t^2 + 5t + 5)^5}$& \href{http://www.lmfdb.org/EllipticCurve/Q/608b1}{\texttt{608b1}} \\
\hline
\texttt{5B.4.2} & $G_{5}$ & 5   & $\smallmat{4}{0}{0}{4},\smallmat{2}{0}{0}{1},\smallmat{1}{1}{0}{1}$ & $  \frac{(t^4+228t^3+494t^2-228t+1)^3}{t(t^2-11t-1)^5}$& \href{http://www.lmfdb.org/EllipticCurve/Q/99d3}{\texttt{99d3}} \\
\hline
\texttt{5B.4.1} & $G_{6}$ & 5   & $\smallmat{4}{0}{0}{4},\smallmat{1}{0}{0}{2},\smallmat{1}{1}{0}{1}$ &  $\frac{(t^4-12t^3+14t^2+12t+1)^3}{t^5(t^2-11t-1)}$& \href{http://www.lmfdb.org/EllipticCurve/Q/99d1}{\texttt{99d1}} \\
\hline
\texttt{5Nn} & $G_{7}$ & 5   & $\smallmat{1}{0}{0}{4},\smallmat{2}{3}{4}{2}$ & $\frac{5^3 (t+1)(2t+1)^3(2t^2-3t+3)^3}{(t^2+t-1)^5}$& \href{http://www.lmfdb.org/EllipticCurve/Q/675b1}{\texttt{675b1}} \\
\hline
\texttt{5B} & $G_{8}$ & 5   & $\smallmat{2}{0}{0}{3},\smallmat{1}{0}{0}{2},\smallmat{1}{1}{0}{1}$ &  $\frac{5^2(t^2+10t+5)^3}{t^5}$  & \href{http://www.lmfdb.org/EllipticCurve/Q/867c1}{\texttt{867c1}} \\
\hline
\texttt{5S4} & $G_{9}$ & 5   & $\smallmat{0}{3}{3}{4},\smallmat{2}{0}{0}{2},\smallmat{3}{0}{4}{4}$ & $t^3(t^2 + 5t + 40)$& \href{http://www.lmfdb.org/EllipticCurve/Q/648a1}{\texttt{648a1}}   \\
\hline
\hline
\texttt{7Ns.3.1} & $G_{1}$ & 7   &  $\smallmat{3}{0}{0}{5},\smallmat{0}{1}{4}{0}$ & $3^3\cdot 5\cdot 7^5/2^7$ & \href{http://www.lmfdb.org/EllipticCurve/Q/2450a1}{\texttt{2450a1}}  \\
\hline 
\texttt{7Ns} & $G_{2}$ & 7    & $\smallmat{0}{6}{1}{0},\smallmat{3}{0}{0}{5},\smallmat{1}{0}{0}{3}$ &  $\frac{t(t + 1)^3(t^2 - 5t + 1)^3(t^2 - 5t + 8)^3(t^4 - 5t^3 + 8t^2 - 7t + 7)^3 }{(t^3 - 4t^2 + 3t + 1)^7}$& \href{http://www.lmfdb.org/EllipticCurve/Q/9225a1}{\texttt{9225a1}}  \\
\hline
\texttt{7B.6.1} & $G_{3}$  & 7    & $\smallmat{6}{0}{0}{6},\smallmat{1}{0}{0}{3},\smallmat{1}{1}{0}{1}$ & $\frac{(t^2 - t + 1)^3(t^6 - 11t^5 + 30t^4 - 15t^3 - 10t^2 + 5t + 1)^3}{(t - 1)^7 t^7(t^3 - 8t^2 + 5t + 1)}$ & \href{http://www.lmfdb.org/EllipticCurve/Q/208d1}{\texttt{208d1}}\\
\hline
\texttt{7B.6.3} & $G_{4}$  & 7    &$\smallmat{6}{0}{0}{6},\smallmat{3}{0}{0}{1},\smallmat{1}{1}{0}{1}$ &   $\frac{(t^2 - t + 1)^3(t^6 + 229t^5 + 270t^4 - 1695t^3 + 1430t^2 - 235t + 1)^3}{(t - 1)t(t^3 - 8t^2 + 5t + 1)^7}$& \href{http://www.lmfdb.org/EllipticCurve/Q/208d2}{\texttt{208d2}}  \\
\hline
\texttt{7B.6.2} & $G_{5}$ & 7     &$\smallmat{6}{0}{0}{6},\smallmat{2}{0}{0}{5},\smallmat{1}{1}{0}{1}$ & $-\frac{(t^2-3t-3)^3(t^2-t+1)^3(3t^2-9t+5)^3(5t^2-t-1)^3}{(t^3-2t^2-t+1) (t^3-t^2-2t+1)^7} $ & \href{http://www.lmfdb.org/EllipticCurve/Q/5733d1}{\texttt{5733d1}}\\
\hline
\texttt{7Nn} & $G_{6}$  & 7    & $\smallmat{1}{0}{0}{6},\smallmat{2}{5}{4}{2}$ & $\frac{64 t^3(t^2+7)^3(t^2-7t+14)^3(5t^2-14t-7)^3}{(t^3-7t^2+7t+7)^7}$& \href{http://www.lmfdb.org/EllipticCurve/Q/15341a1}{\texttt{15341a1}}\\
\hline
\texttt{7B} & $G_{7}$  & 7   &$\smallmat{3}{0}{0}{5},\smallmat{1}{0}{0}{3},\smallmat{1}{1}{0}{1}$ &   $\frac{(t^2+245t+2401)^3(t^2+13t+49)}{t^7}$& \href{http://www.lmfdb.org/EllipticCurve/Q/338a1}{\texttt{338a1}}\\
\hline
\hline
\texttt{11B.10.4} & $G_{1}$ & 11    & $\smallmat{10}{0}{0}{10},\smallmat{4}{0}{0}{6},\smallmat{1}{1}{0}{1}$ & $-11^2$& \href{http://www.lmfdb.org/EllipticCurve/Q/1089f2}{\texttt{1089f2}} \\
\hline
\texttt{11B.10.5} & $G_{2}$& 11   &$\smallmat{10}{0}{0}{10},\smallmat{5}{0}{0}{7},\smallmat{1}{1}{0}{1}$ &  $-11\cdot 131^3$ & \href{http://www.lmfdb.org/EllipticCurve/Q/1089f1}{\texttt{1089f1}} \\
\hline
\texttt{11Nn} & $G_{3}$& 11   & $\smallmat{1}{0}{0}{10},\smallmat{3}{5}{8}{3}$ &  $\frac{P_{11}(x,y)^3}{(11y+(2x^2+17x-34))^2( (x-4)y-(5x-9))^{11}}\,,\quad y^2+y = x^3-x^2-7x+10 $ & \href{http://www.lmfdb.org/EllipticCurve/Q/232544f1}{\texttt{232544f1}}  \\
\hline
\hline
\texttt{13B.5.2} & $G_{1}$ & 13   &$\smallmat{5}{0}{0}{8},\smallmat{2}{0}{0}{1},\smallmat{1}{1}{0}{1}$  & $ \frac{(t^2 - t + 1)^3P_{13}(t)^3}{(t - 1)t(t^3 - 4t^2 + t + 1)^{13}}$ & \href{http://www.lmfdb.org/EllipticCurve/Q/2890d2}{\texttt{2890d2}}  \\
\hline
\texttt{13B.5.1} & $G_{2}$ & 13  & $\smallmat{5}{0}{0}{8},\smallmat{1}{0}{0}{2},\smallmat{1}{1}{0}{1}$ &  $\frac{(t^2 - t + 1)^3(t^{12} - 9t^{11} + 29t^{10} - 40t^9 + 22t^8 - 16t^7 + 40t^6 - 22t^5 - 23t^4 + 25t^3 - 4t^2 - 3t + 1)^3}{(t - 1)^{13} t^{13} (t^3 - 4t^2 + t + 1)}$& \href{http://www.lmfdb.org/EllipticCurve/Q/2890d1}{\texttt{2890d1}}\\
\hline
\texttt{13B.5.4} & $G_{3}$  & 13 &  $\smallmat{5}{0}{0}{8},\smallmat{4}{0}{0}{7},\smallmat{1}{1}{0}{1}$ &  $-\frac{13^4(t^2-t+1)^3 (t^4-t^3+2t^2-9t+3)^3(3t^4-3t^3-7t^2+12t-4)^3(4t^4-4t^3-5t^2+3t-1)^3}{(t^3-4t^2+t+1)^{13} (5t^3-7t^2-8t+5)} $ & \href{http://www.lmfdb.org/EllipticCurve/Q/216320i1}{\texttt{216320i1}} \\
\hline
\texttt{13B.4.2} & $G_{4}$  & 13 & $\smallmat{4}{0}{0}{10},\smallmat{2}{0}{0}{1},\smallmat{1}{1}{0}{1}$ &  $\frac{(t^4 - t^3 + 5t^2 + t + 1)(t^8 + 235t^7 + 1207t^6 + 955t^5 + 3840t^4 - 955t^3 + 1207t^2 - 235t+ 1)^3}{t (t^2 - 3t - 1)^{13}} $ & \href{http://www.lmfdb.org/EllipticCurve/Q/147c2}{\texttt{147c2}} \\
\hline
\texttt{13B.4.1} & $G_{5}$  & 13 &$\smallmat{4}{0}{0}{10},\smallmat{1}{0}{0}{2},\smallmat{1}{1}{0}{1}$ &  $ \frac{(t^4 - t^3 + 5t^2 + t + 1) (t^8 - 5t^7 + 7t^6 - 5t^5 + 5t^3 + 7t^2 + 5t + 1)^3}{t^{13} (t^2 - 3t - 1)}$ & \href{http://www.lmfdb.org/EllipticCurve/Q/147c1}{\texttt{147c1}} \\
\hline
\texttt{13B} & $G_{6}$  & 13 & $\smallmat{2}{0}{0}{7},\smallmat{1}{0}{0}{2},\smallmat{1}{1}{0}{1}$ &  $\frac{(t^2+5t+13)(t^4+7t^3+20t^2+19t+1)^3}{t}$& \href{http://www.lmfdb.org/EllipticCurve/Q/2450bb1}{\texttt{2450bb1}}  \\
\hline
 \texttt{13S4} &$G_{7}$  & 13 & $\smallmat{3}{0}{12}{9},\smallmat{2}{0}{0}{2},\smallmat{9}{5}{0}{6}$ &    \color{black} $\frac{2^4\cdot 5\cdot 13^4\cdot 17^3}{3^{13}},-\frac{2^{12}\cdot 5^3\cdot 11\cdot 13^4}{3^{13}} \,\,\text{ or }\,\,\frac{2^{18}\cdot3^3\cdot 13^4\cdot 127^3\cdot 139^3\cdot 157^3\cdot 283^3\cdot 929}{5^{13}\cdot 61^{13}}$& \href{http://www.lmfdb.org/EllipticCurve/Q/50700l1}{\texttt{50700l1}}   \\
\hline
\hline
 \texttt{17B.4.2} & $G_{1}$  & 17 & $\smallmat{4}{0}{0}{13},\smallmat{2}{0}{0}{10},\smallmat{1}{1}{0}{1}$ &  $ -17\cdot 373^3/2^{17}$ & \href{http://www.lmfdb.org/EllipticCurve/Q/14450bk1}{\texttt{14450bk1}} \\
\hline
\texttt{17B.4.6} & $G_{2}$   & 17& $\smallmat{4}{0}{0}{13},\smallmat{6}{0}{0}{9},\smallmat{1}{1}{0}{1}$ &  $-17^2 \cdot 101^3/2 $ & \href{http://www.lmfdb.org/EllipticCurve/Q/14450bk2}{\texttt{14450bk2}} \\
\hline
\hline
\texttt{37B.8.1} & $G_{1}$   & 37& $\smallmat{8}{0}{0}{14},\smallmat{1}{0}{0}{2},\smallmat{1}{1}{0}{1}$ &  $-7\cdot 11^3 $ & \href{http://www.lmfdb.org/EllipticCurve/Q/1225e1}{\texttt{1225e1}}\\
\hline
\texttt{37B.8.2} & $G_{2}$  & 37 & $\smallmat{8}{0}{0}{14},\smallmat{2}{0}{0}{1},\smallmat{1}{1}{0}{1}$ &  $-7\cdot 137^3\cdot 2083^3 $ & \href{http://www.lmfdb.org/EllipticCurve/Q/1225e2}{\texttt{1225e2}} \\
 \hline
\end{tabular}
\end{footnotesize}
\caption{Groups $G_E(p)$ containing $-I$, for non-CM elliptic curves $E/\Q$}\label{jmaps}
\vspace{10pt}
\end{table}
}


{\scriptsize
\begin{center}
\begin{tabular}{ccc}
\begin{tabular}{|c|c|c|c|}

\hline
Sutherland &   Zywina & Model &  Example\\
\hline
\texttt{ 3B.1.1 } & $H_{3,1}$ &  $\mathcal{E}_{3,3}$ & \href{http://www.lmfdb.org/EllipticCurve/Q/19a3}{\texttt{19a3}}  \\ 
\texttt{ 3B.1.2 } & $H_{3,2}$ & $\mathcal{E}^{-3}_{3,3}$ & \href{http://www.lmfdb.org/EllipticCurve/Q/19a2}{\texttt{19a2}}  \\ 
\texttt{ 3Cs.1.1 } &  $H_{1,1}$ & $\mathcal{E}_{3,1}$ or $\mathcal{E}^{-3}_{3,1}$ & \href{http://www.lmfdb.org/EllipticCurve/Q/19a1}{\texttt{19a1}}  \\ 
\hline
\texttt{ 5B.1.1 } & $H_{6,1}$ & $\mathcal{E}_{5,6}$ & \href{http://www.lmfdb.org/EllipticCurve/Q/11a3}{\texttt{11a3}}  \\ 
\texttt{ 5B.1.2 } &  $H_{5,1}$ & $\mathcal{E}_{5,5}$ &\href{http://www.lmfdb.org/EllipticCurve/Q/11a2}{\texttt{11a2}}  \\ 
\texttt{ 5B.1.3 } &  $H_{5,2}$ & $\mathcal{E}^{5}_{5,5}$ & \href{http://www.lmfdb.org/EllipticCurve/Q/75a1}{\texttt{75a1}}  \\ 
\texttt{ 5B.1.4 } &  $H_{6,2}$ &$\mathcal{E}^{5}_{5,6}$ & \href{http://www.lmfdb.org/EllipticCurve/Q/75a2}{\texttt{75a2}}  \\ 
\texttt{ 5Cs.1.1 } &  $H_{1,1}$ & $\mathcal{E}_{5,1}$ &\href{http://www.lmfdb.org/EllipticCurve/Q/11a1}{\texttt{11a1}}  \\ 
\texttt{ 5Cs.1.3 } &  $H_{1,2}$ & $\mathcal{E}^{5}_{5,1}$ & \href{http://www.lmfdb.org/EllipticCurve/Q/275b2}{\texttt{275b2}}  \\ 
\hline
\end{tabular}
& 
\begin{tabular}{|c|c|c|c|}
\hline
Sutherland &   Zywina &  Model & Example\\
\hline
\texttt{ 7B.1.1 } &  $H_{3,1}$ & $\mathcal{E}_{7,3}$  & \href{http://www.lmfdb.org/EllipticCurve/Q/26b1}{\texttt{26b1}}  \\ 
\texttt{ 7B.1.2 } &  $H_{5,2}$ & $\mathcal{E}^{-7}_{7,5}$ & \href{http://www.lmfdb.org/EllipticCurve/Q/637a1}{\texttt{637a1}}  \\ 
\texttt{ 7B.1.3 } &  $H_{4,1}$ & $\mathcal{E}_{7,4}$ & \href{http://www.lmfdb.org/EllipticCurve/Q/26b2}{\texttt{26b2}}  \\ 
\texttt{ 7B.1.4 } &  $H_{4,2}$ & $\mathcal{E}^{-7}_{7,4}$ & \href{http://www.lmfdb.org/EllipticCurve/Q/294a1}{\texttt{294a1}}  \\ 
\texttt{ 7B.1.5 } &  $H_{5,1}$ & $\mathcal{E}_{7,5}$ & \href{http://www.lmfdb.org/EllipticCurve/Q/637a2}{\texttt{637a2}}  \\ 
\texttt{ 7B.1.6 } &  $H_{3,2}$ & $\mathcal{E}^{-7}_{7,3}$ & \href{http://www.lmfdb.org/EllipticCurve/Q/294a2}{\texttt{294a2}}  \\ 
\texttt{ 7B.2.1 } &  $H_{7,2}$ & $\mathcal{E}^{-7}_{7,7}$ & \href{http://www.lmfdb.org/EllipticCurve/Q/338b1}{\texttt{338b1}}  \\ 
\texttt{ 7B.2.3 } &  $H_{7,1}$ & $\mathcal{E}_{7,7}$  & \href{http://www.lmfdb.org/EllipticCurve/Q/338b2}{\texttt{338b2}}  \\ 
\texttt{ 7Ns.2.1 } &  $H_{1,1}$ & $\mathcal{E}_{7,1}$ or $\mathcal{E}^{-7}_{7,1}$ & \href{http://www.lmfdb.org/EllipticCurve/Q/2450ba1}{\texttt{2450ba1}}  \\ 
\hline
\end{tabular}
& 
\begin{tabular}{|c|c|c|c|}
\hline
Sutherland &   Zywina &  Model & Example\\
\hline
\texttt{ 11B.1.4 } &  $H_{1,1}$ & $\mathcal{E}_{11,1}$ &   \href{http://www.lmfdb.org/EllipticCurve/Q/121a2}{\texttt{121a2}} \\ 
\texttt{ 11B.1.5 } &  $H_{2,1}$&  $\mathcal{E}_{11,2}$ &\href{http://www.lmfdb.org/EllipticCurve/Q/121a1}{\texttt{121a1}}  \\ 
\texttt{ 11B.1.6 } &  $H_{2,2}$& $\mathcal{E}^{-11}_{11,2}$  &   \href{http://www.lmfdb.org/EllipticCurve/Q/121c2}{\texttt{121c2}}\\ 
\texttt{ 11B.1.7 } &  $H_{1,2}$& $\mathcal{E}^{-11}_{11,1}$ & \href{http://www.lmfdb.org/EllipticCurve/Q/121c1}{\texttt{121c1}}  \\ 
\hline

\texttt{ 13B.3.1 } &  $H_{5,1}$& $\mathcal{E}_{13,5}$ & \href{http://www.lmfdb.org/EllipticCurve/Q/147b1}{\texttt{147b1}}  \\ 
\texttt{ 13B.3.2 } &  $H_{4,1}$& $\mathcal{E}_{13,4}$ & \href{http://www.lmfdb.org/EllipticCurve/Q/147b2}{\texttt{147b2}}  \\ 
\texttt{ 13B.3.4 } &  $H_{5,2}$& $\mathcal{E}^{13}_{13,5}$ & \href{http://www.lmfdb.org/EllipticCurve/Q/24843o1}{\texttt{24843o1}}  \\ 
\texttt{ 13B.3.7 } &  $H_{4,2}$ & $\mathcal{E}^{13}_{13,4}$ & \href{http://www.lmfdb.org/EllipticCurve/Q/24843o2}{\texttt{24843o2}}  \\ 
\hline
\end{tabular}
\end{tabular}
\captionof{table}{Groups $G_E(p)$ not containing $-I$, for non-CM elliptic curves $E/\Q$}\label{table_mod_single}
\end{center}
\begin{center}
\begin{tabular}{|rl|}
\hline
$\mathcal{E}_{3,1}\,:\,y^2$&$= x^3 -3(t+1)(t+3)(t^2+3) x -2(t^2-3)(t^4+6t^3+18t^2+18t+9)$\\
$\mathcal{E}_{3,3}\,:\, y^2$&$= x^3 -3(t+1)^3(t+9)x -2(t+1)^4(t^2-18t-27)$\\
\hline
$\mathcal{E}_{5,1}\,:\, y^2$& $=x^3 -27(t^{20} + 228t^{15} + 494t^{10} - 228t^5 + 1)x +54(t^{30} - 522t^{25} - 10005t^{20} - 10005t^{10} + 522t^5 + 1)$\\
$\mathcal{E}_{5,5}\,:\, y^2$&$=x^3-27(t^4 + 228t^3 + 494t^2 - 228t + 1)x+54(t^6 - 522t^5 - 10005t^4 - 10005t^2 + 522t + 1)$\\
$\mathcal{E}_{5,6}\,:\, y^2$& $= x^3-27(t^4 - 12t^3 + 14t^2 + 12t + 1)x + 54(t^6-18t^5+75t^4+75t^2+18t+1)$\\
\hline
$\mathcal{E}_{7,1}\,:\, y^2$&$= x^3 -5^3  7^3 x - 5^4  7^2 106$\\
$\mathcal{E}_{7,3}\,:\, y^2$&$=x^3 - 27(t^2 - t + 1)(t^6 - 11t^5 + 30t^4 - 15t^3 - 10t^2 + 5t + 1) x+54Q_{7,3}(t)$\\
$\mathcal{E}_{7,4}\,:\, y^2$&$=x^3 -27(t^2-t+1)(t^6+229t^5+270t^4-1695t^3+1430t^2-235t+1) x +54Q_{7,4}(t)$\\
$\mathcal{E}_{7,5}\,:\, y^2$&$= x^3 -   27\cdot 7 (t^2 - 3t - 3)(t^2 - t + 1)(3t^2 - 9t + 5)(5t^2 - t - 1) x -54\cdot 7^2 Q_{7,5}(t)$\\
$\mathcal{E}_{7,7}\,:\, y^2$& $=x^3 - 27(t^2 + 13t + 49)^3(t^2 + 245t + 2401)x +54(t^2 + 13t + 49)^4(t^4 - 490t^3 - 21609t^2 - 235298t - 823543)$\\
\hline
$\mathcal{E}_{11,1}\,:\,  y^2 $&$ = x^3-27\cdot 11^4 x + 54\cdot 11^5\cdot 43$\\
$\mathcal{E}_{11,2}\,:\,  y^2 $& $=x^3-27\cdot 11^3\cdot 131 x +54\cdot 11^4\cdot 4973$\\
\hline
$\mathcal{E}_{13,4}\,:\, y^2 $& $= x^3-27(t^4 - t^3 + 5t^2 + t + 1)^3P_{13,4}(t) x + 54(t^2 + 1) (t^4 - t^3 + 5 t^2 + t + 1)^4 Q_{13,4}(t)$\\
$\mathcal{E}_{13,5}\,:\, y^2 $& $= x^3 -27(t^4 - t^3 + 5t^2 + t + 1)^3P_{13,5}(t) x + 54(t^2 + 1) (t^4 - t^3 + 5 t^2 + t + 1)^4 Q_{13,5}(t)$\\
\hline
\end{tabular}
\captionof{table}{Elliptic curve model for the fine moduli}\label{Eq_finemoduli}
\end{center}

}


{
\extrarowheight = -0.5ex
\renewcommand{\arraystretch}{2}
\begin{center}
\begin{tabular}{|c|c|c|c|c|}
\hline
Type & level & generators & $j$-map & Example\\
\hline
\texttt{4X3} & 4 &
$\smallmat{3}{3}{0}{1}, \smallmat{0}{1}{3}{1}$
& $-t^2+1728$ & \href{http://www.lmfdb.org/EllipticCurve/Q/567a1}{\texttt{567a1}}\\
\hline
\texttt{4X7} & 4 &  
$\smallmat{0}{1}{3}{0}, \smallmat{0}{1}{1}{1}$
& $-4t^3(t+8)$ & \href{http://www.lmfdb.org/EllipticCurve/Q/216a1}{\texttt{216a1}}\\
\hline\texttt{8X4} & 8 &
$\smallmat{7}{7}{0}{1}, \smallmat{5}{0}{1}{1}$
& $-2t^2+1728$ & \href{http://www.lmfdb.org/EllipticCurve/Q/216a1}{\texttt{216a1}}\\
\hline
\texttt{8X5} & 8 & 
$\smallmat{5}{5}{0}{1}, \smallmat{0}{1}{1}{1}$
& $2t^2+1728$ &  \href{http://www.lmfdb.org/EllipticCurve/Q/1682b1}{\texttt{1682b1}}\\
\hline
\texttt{9XE} & 9 &
$\smallmat{4}{5}{4}{4}, \smallmat{4}{8}{8}{6}$
& $\frac{3^7 (t^2-1)^3 (t^6+3t^5+6t^4+t^3-3t^2+12t+16)^3 (2t^3+3t^2-3t-5)}{(t^3-3t-1)^9}$ &  \href{http://www.lmfdb.org/EllipticCurve/Q/1944c1}{\texttt{1944c1}}\\
\hline
\end{tabular}
\captionof{table}{Maximal groups $G^\infty_E(p)$, for non-CM elliptic curves $E/\Q$ and $p=2,3$}\label{tab:adic_groups}
\end{center}
}

\begin{center}
\begin{tabular}{|c|c|c|}
\hline
Type & $j$-invariants & Examples \\
\hline
\texttt{[4X3,11B.10.4]} & $-11^2$ &   \href{http://www.lmfdb.org/EllipticCurve/Q/1089h1}{\texttt{1089h1}}  \\ 
\hline
\texttt{[4X3,11B.10.5]} &  $-11\cdot 131^3$ &   \href{http://www.lmfdb.org/EllipticCurve/Q/1089h2}{\texttt{1089h2}} \\ 
\hline
\texttt{[8X5,7Ns.3.1]} & $3^3\cdot 5\cdot 7^5/2^7$ &   \href{http://www.lmfdb.org/EllipticCurve/Q/2450a1}{\texttt{2450a1}} \\ 
\hline
\end{tabular}
\captionof{table}{Modular curves: isolated point}\label{tab:jconstant}
\end{center}

\begin{center}
\begin{tabular}{|c|c|c|}
\hline 
Type & $j$-maps & Example\\
\hline
\texttt{[2B,3B]} & $\frac{(11 t-8)^3 \left(1259 t^3-2856 t^2+2112 t-512\right)^3}{2 (t-1) t^6 (3 t-2)^3 (25 t-16)^2}$ &  \href{http://www.lmfdb.org/EllipticCurve/Q/80b1}{\texttt{80b1}}  \\ 
\hline
\texttt{[2B,3Cs]} & $-\frac{\left(54 t^3-1\right)^3 \left(54 t^3-54 t^2-1\right)^3 \left(2916 t^6+2916 t^5+2916 t^4-108 t^3-54 t^2+1\right)^3}{729 t^6 (3 t+1)^6 (6 t-1)^3 \left(9 t^2-3 t+1\right)^6 \left(36 t^2+6 t+1\right)^3}$ & \href{http://www.lmfdb.org/EllipticCurve/Q/98a3}{\texttt{98a3}}  \\ 
\hline
\texttt{[2B,3Nn]} & $\frac{\left(t^3+16\right)^3}{t^3}$ & \href{http://www.lmfdb.org/EllipticCurve/Q/1568d1}{\texttt{1568d1}}  \\ 
\hline
\texttt{[2B,3Ns]} & $\frac{\left(1024 t^3+1920 t^2+768 t+115\right)^3 \left(1024 t^3+1920 t^2+1200 t+223\right)^3}{46656 (t+1)^3 (4 t+1)^3 (16 t+7)^6}$ & \href{http://www.lmfdb.org/EllipticCurve/Q/726a1}{\texttt{726a1}}  \\ 
\hline
\texttt{[2B,5B]} & $-\frac{\left(5 t^6-2080 t^5+81920 t^4-1310720 t^3+10485760 t^2-41943040 t+67108864\right)^3}{8 (t-8)^5 t^{10} (5 t-32)^2}$ & \href{http://www.lmfdb.org/EllipticCurve/Q/768b1}{\texttt{768b1}}  \\ 
\hline
 \texttt{[2B,5B.4.1]} & $\frac{P_1(t)^3}{590490000000000 t^{10} (3788 t+1)^5 (3818 t+1)^{10} (3848 t+1)^5 \left(14118064 t^2+7516 t+1\right) \left(14690764 t^2+7666 t+1\right)^2}$  & \href{http://www.lmfdb.org/EllipticCurve/Q/198e1}{\texttt{198e1}}  \\ 
\hline
\texttt{[2B,5B.4.2]} & $-\frac{P_2(t)^3}{4 t^2 (2 t-1)^2 (4 t-1) \left(4 t^2+2 t-1\right)^{10} \left(16 t^2-12 t+1\right)^5}$ & \href{http://www.lmfdb.org/EllipticCurve/Q/198e3}{\texttt{198e3}}  \\ 
\hline
\hline
\texttt{[2Cn,3B]} & $\frac{9 \left(t^2+3\right) \left(t^2+27\right)^3}{t^6}$ & \href{http://www.lmfdb.org/EllipticCurve/Q/196a1}{\texttt{196a1}}  \\ 
\hline
\texttt{[2Cn,5S4]} & $\frac{\left(3 t^2+1\right)^3 \left(64 t^4+11 t^2+1\right)}{t^{10}}$ & \href{http://www.lmfdb.org/EllipticCurve/Q/1444a1}{\texttt{1444a1}}  \\ 
\hline
\texttt{[2Cn,7B]} & $\frac{\left(7 t^2-t+1\right) \left(7 t^2+t+1\right) \left(2401 t^4+245 t^2+1\right)^3}{t^2}$ & \href{http://www.lmfdb.org/EllipticCurve/Q/1922e1}{\texttt{1922e1}}  \\ 
\hline
\hline
\texttt{[2Cs,3B]} & $ \frac{\left(7 t^2+6 t+3\right)^3 \left(127 t^6+738 t^5+1605 t^4+1260 t^3+345 t^2+18 t+3\right)^3}{4 (t-1)^6 t^2 (t+1)^6 (t+3)^2 (3 t+1)^6 (5 t+3)^2}$  & \href{http://www.lmfdb.org/EllipticCurve/Q/150c2}{\texttt{150c2}}  \\ 
\hline
\hline
\texttt{[3Nn,5B]} & $ \frac{\left(3125 t^6+250 t^3+1\right)^3}{t^3}$ & \href{http://www.lmfdb.org/EllipticCurve/Q/1369e1}{\texttt{1369e1}}\\
\hline
\hline
\texttt{[4X3,3B]} & $-\frac{9 \left(t^2-2 t-26\right)^3 \left(t^2-2 t-2\right)}{(t-1)^6}$ & \href{http://www.lmfdb.org/EllipticCurve/Q/242a1}{\texttt{242a1}} \\
\hline
\texttt{[4X3,5S4]} &$\frac{\left(3 t^2-1\right)^3 \left(64 t^4-11 t^2+1\right)}{t^{10}}$ & \href{http://www.lmfdb.org/EllipticCurve/Q/324b1}{\texttt{324b1}} \\
\hline
\texttt{[4X3,7B]} & $-\frac{\left(t^4-245 t^2+2401\right)^3 \left(t^4-13 t^2+49\right)}{t^{14}}$& \href{http://www.lmfdb.org/EllipticCurve/Q/1369c2}{\texttt{1369c2}} \\
\hline
\hline
\texttt{[4X7,3Nn]} & $-\frac{\left(32 t^3+1\right)^3}{64 t^{12}}$ & \href{http://www.lmfdb.org/EllipticCurve/Q/80802b1}{\texttt{80802b1}} \\
\hline
\hline
\texttt{[8X4,3B]} & $-\frac{9 \left(t^2-2 t-53\right)^3 \left(t^2-2 t-5\right)}{2 (t-1)^6}$ & \href{http://www.lmfdb.org/EllipticCurve/Q/1296k2}{\texttt{1296k2}} \\
\hline
\texttt{[8X4,5S4]} &$\frac{2 \left(3 t^2-2\right)^3 \left(32 t^4-11 t^2+2\right)}{t^{10}}$ & \href{http://www.lmfdb.org/EllipticCurve/Q/4232d1}{\texttt{4232d1}} \\
\hline
\texttt{[8X4,7B]} &$-\frac{\left(49 t^4-26 t^2+4\right) \left(2401 t^4-490 t^2+4\right)^3}{128 t^2}$ & \href{http://www.lmfdb.org/EllipticCurve/Q/162c3}{\texttt{162c3}} \\
\hline
\hline
\texttt{[8X5,3B]} & $\frac{9 \left(t^2+6\right) \left(t^2+54\right)^3}{2 t^6}$& \href{http://www.lmfdb.org/EllipticCurve/Q/7938d1}{\texttt{7938d1}} \\
\hline
\texttt{[8X5,5S4]} & $\frac{2 \left(3 t^2+2\right)^3 \left(32 t^4+11 t^2+2\right)}{t^{10}}$& \href{http://www.lmfdb.org/EllipticCurve/Q/16200e1}{\texttt{16200e1}} \\
\hline
\texttt{[8X5,7B]} & $\frac{\left(49 t^4+26 t^2+4\right) \left(2401 t^4+490 t^2+4\right)^3}{128 t^2}$& \href{http://www.lmfdb.org/EllipticCurve/Q/12482f2}{\texttt{12482f2}}\\
\hline
\end{tabular}
\captionof{table}{Modular curves of genus 0}\label{jmapsj0coarse}
\end{center}

\begin{center}
\begin{tabular}{|c|c|c|c|}
\hline 
Type & Model & Parametrization & Example\\
\hline
  \texttt{[2B,3B.1.1]} & $\mathcal{E}_{3,3}$ & \multirow{ 2}{*}{$\frac{512 (t-1) (3 t-2)^3}{t^3 (25 t-16)}$} &  \href{http://www.lmfdb.org/EllipticCurve/Q/14a4}{\texttt{14a4}}  \\ 
\cline{1-2}\cline{4-4}
  \texttt{[2B,3B.1.2]} & $\mathcal{E}^{-3}_{3,3}$ & &  \href{http://www.lmfdb.org/EllipticCurve/Q/14a3}{\texttt{14a3}}  \\ 
\hline
  \texttt{[2B,3Cs.1.1]} &$\mathcal{E}_{3,1}$ or $\mathcal{E}^{-3}_{3,1}$ &$-\frac{(6 t-1)^3 \left(36 t^2+6 t+1\right)^3}{432 t^3 (3 t+1)^3 \left(9 t^2-3 t+1\right)^3}$ & \href{http://www.lmfdb.org/EllipticCurve/Q/14a1}{\texttt{14a1}}  \\ 
\hline
 \texttt{[2B,5B.1.1]} & $\mathcal{E}_{5,6}$ & \multirow{ 2}{*}{$\frac{(3788 t+1)^5 (3848 t+1)^5 \left(14118064 t^2+7516 t+1\right)}{388800000 t^5 (3818 t+1)^5 \left(14690764 t^2+7666 t+1\right)}$} & \href{http://www.lmfdb.org/EllipticCurve/Q/66c1}{\texttt{66c1}}  \\ 
\cline{1-2}\cline{4-4}
 \texttt{[2B,5B.1.4]} & $\mathcal{E}^{5}_{5,6}$ & & \href{http://www.lmfdb.org/EllipticCurve/Q/150b3}{\texttt{150b3}}  \\ 
\hline
 \texttt{[2B,5B.1.2]} &$\mathcal{E}_{5,5}$ &\multirow{ 2}{*}{$-\frac{(4 t-1) \left(16 t^2-12 t+1\right)^5}{32 t (2 t-1) \left(4 t^2+2 t-1\right)^5}$}  & \href{http://www.lmfdb.org/EllipticCurve/Q/66c3}{\texttt{66c3}}  \\ 
\cline{1-2}\cline{4-4}
 \texttt{[2B,5B.1.3]} &$\mathcal{E}^{5}_{5,5}$ & & \href{http://www.lmfdb.org/EllipticCurve/Q/150b1}{\texttt{150b1}}  \\ 
\hline
 \texttt{[2Cn,3B.1.1]} &$\mathcal{E}_{3,3}$ &\multirow{ 2}{*}{$\frac{3 \left(t^4-54 t^2-243\right)}{t^3}$} & \href{http://www.lmfdb.org/EllipticCurve/Q/196b1}{\texttt{196b1}}  \\ 
\cline{1-2}\cline{4-4}
 \texttt{[2Cn,3B.1.2]} &$\mathcal{E}^{-3}_{3,3}$ & & \href{http://www.lmfdb.org/EllipticCurve/Q/196b2}{\texttt{196b2}}  \\ 
\hline
 \texttt{[2Cn,7B.2.1]} &$\mathcal{E}^{-7}_{7,7}$ & \multirow{ 2}{*}{$-\frac{823543 t^8+235298 t^6+21609 t^4+490 t^2-1}{t}$}& \href{http://www.lmfdb.org/EllipticCurve/Q/1922c1}{\texttt{1922c1}}  \\ 
\cline{1-2}\cline{4-4}
 \texttt{[2Cn,7B.2.3]} &$\mathcal{E}_{7,7}$ & & \href{http://www.lmfdb.org/EllipticCurve/Q/1922c2}{\texttt{1922c2}}  \\ 
\hline
 \texttt{[2Cs,3B.1.1]} &$\mathcal{E}_{3,3}$ & \multirow{ 2}{*}{$-\frac{(t+3) (3 t+1)^3}{32 t (t+1)^3}$}& \href{http://www.lmfdb.org/EllipticCurve/Q/30a2}{\texttt{30a2}}  \\ 
\cline{1-2}\cline{4-4}
 \texttt{[2Cs,3B.1.2]} &$\mathcal{E}^{-3}_{3,3}$ & & \href{http://www.lmfdb.org/EllipticCurve/Q/30a6}{\texttt{30a6}}  \\ 
\hline
\end{tabular}
\captionof{table}{Modular curves of genus 0: Fine Moduli Spaces}\label{jmapsj0fine}
\end{center}


{\scriptsize
\begin{center}
\begin{tabular}{ccc}
\begin{tabular}{|c|c|c|}
\hline
Type & $\mathcal E/\Q$ & $\mathcal E(\Q)$ \\ 
\hline
\texttt{[2B,5Nn]} & \href{http://www.lmfdb.org/EllipticCurve/Q/50b1}{\texttt{50b1}} & $\Z/5\Z$\\ \hline
\texttt{[2B,5Ns]} & \href{http://www.lmfdb.org/EllipticCurve/Q/50b1}{\texttt{50b1}} & $\Z/5\Z$\\ \hline
\texttt{[2B,5S4]} & \href{http://www.lmfdb.org/EllipticCurve/Q/50b2}{\texttt{50b2}} & $\Z/5\Z$\\ \hline
\texttt{[2Cn,13B]} & \href{http://www.lmfdb.org/EllipticCurve/Q/52a1}{\texttt{52a1}} & $\Z/2\Z$\\ \hline
\texttt{[2Cn,13B.4.1]} & \href{http://www.lmfdb.org/EllipticCurve/Q/52a2}{\texttt{52a2}} & $\Z/2\Z$\\ \hline
\texttt{[2Cn,13B.4.2]} & \href{http://www.lmfdb.org/EllipticCurve/Q/52a2}{\texttt{52a2}} & $\Z/2\Z$\\ \hline
\texttt{[2Cn,3Cs]} & \href{http://www.lmfdb.org/EllipticCurve/Q/36a3}{\texttt{36a3}} & $\Z/2\Z$\\ \hline
\texttt{[2Cn,3Nn]} & \href{http://www.lmfdb.org/EllipticCurve/Q/36a3}{\texttt{36a3}} & $\Z/2\Z$\\ \hline
\texttt{[2Cn,3Ns]} & \href{http://www.lmfdb.org/EllipticCurve/Q/36a4}{\texttt{36a4}} & $\Z/2\Z$\\ \hline
\texttt{[2Cn,5B]} & \href{http://www.lmfdb.org/EllipticCurve/Q/20a3}{\texttt{20a3}} & $\Z/2\Z$\\ \hline
\texttt{[2Cn,5B.4.1]} & \href{http://www.lmfdb.org/EllipticCurve/Q/20a4}{\texttt{20a4}} & $\Z/2\Z$\\ \hline
\texttt{[2Cn,5B.4.2]} & \href{http://www.lmfdb.org/EllipticCurve/Q/20a4}{\texttt{20a4}} & $\Z/2\Z$\\ \hline
\texttt{[2Cn,5Nn]} & \multicolumn{2}{c|}{Non Elliptic: $\mathcal E(\Q_5)=\emptyset$}  \\ \hline
\end{tabular}
&
\begin{tabular}{|c|c|c|}
\hline
Type & $\mathcal E/\Q$ & $\mathcal E(\Q)$ \\ 
\hline
\texttt{[2Cs,3Nn]} & \href{http://www.lmfdb.org/EllipticCurve/Q/36a1}{\texttt{36a1}} & $\Z/6\Z$\\ \hline
\texttt{[2Cs,3Ns]} & \href{http://www.lmfdb.org/EllipticCurve/Q/36a1}{\texttt{36a1}} & $\Z/6\Z$\\ \hline 
\texttt{[3B,5B]} & \href{http://www.lmfdb.org/EllipticCurve/Q/15a1}{\texttt{15a1}} & $\Z/2\Z\oplus \Z/4\Z$\\ \hline
\texttt{[3B,5B.4.1]} & \href{http://www.lmfdb.org/EllipticCurve/Q/15a3}{\texttt{15a3}} & $\Z/2\Z\oplus \Z/4\Z$\\ \hline
\texttt{[3B,5B.4.2]} & \href{http://www.lmfdb.org/EllipticCurve/Q/15a3}{\texttt{15a3}} & $\Z/2\Z\oplus \Z/4\Z$\\ \hline
\texttt{[3B,5S4]} & \href{http://www.lmfdb.org/EllipticCurve/Q/75c1}{\texttt{75c1}} & $\Z/5\Z$\\ \hline
\texttt{[3B,7B]} & \href{http://www.lmfdb.org/EllipticCurve/Q/21a1}{\texttt{21a1}} & $\Z/2\Z\oplus \Z/4\Z$\\ \hline
\texttt{[3Nn,5Nn]} & \href{http://www.lmfdb.org/EllipticCurve/Q/225a1}{\texttt{225a1}} & $\Z$\\ \hline
\texttt{[3Nn,5Ns]} & \href{http://www.lmfdb.org/EllipticCurve/Q/225a1}{\texttt{225a1}} & $\Z$\\ \hline
\texttt{[3Nn,5S4]} & \href{http://www.lmfdb.org/EllipticCurve/Q/225a2}{\texttt{225a2}} & $\Z$\\ \hline
\texttt{[3Nn,7Nn]} & \href{http://www.lmfdb.org/EllipticCurve/Q/441b1}{\texttt{441b1}} & $\Z/3\Z\oplus \Z$\\ \hline
\texttt{[3Ns,5B]} & \href{http://www.lmfdb.org/EllipticCurve/Q/15a3}{\texttt{15a3}} & $\Z/2\Z\oplus \Z/4\Z$\\ \hline
\end{tabular}
&
\begin{tabular}{|c|c|c|}
\hline
Type & $\mathcal E/\Q$ & $\mathcal E(\Q)$ \\ 
\hline
\texttt{[4X3,3Nn]} & \href{http://www.lmfdb.org/EllipticCurve/Q/144a3}{\texttt{144a3}} & $\Z/2\Z$\\ \hline
\texttt{[4X3,13B]} & \href{http://www.lmfdb.org/EllipticCurve/Q/208c1}{\texttt{208c1}} & $\Z/2\Z$\\ \hline
\texttt{[4X3,5B]} & \href{http://www.lmfdb.org/EllipticCurve/Q/80b3}{\texttt{80b3}} & $\Z/2\Z$\\ \hline
\texttt{[4X3,5Nn]} &  \multicolumn{2}{c|}{Non Elliptic: $\mathcal E(\Q_5)=\emptyset$}  \\ \hline
\texttt{[4X7,3B]} & \href{http://www.lmfdb.org/EllipticCurve/Q/48a6}{\texttt{48a6}} & $\Z/8\Z$\\ \hline
\texttt{[4X7,5B]} & \href{http://www.lmfdb.org/EllipticCurve/Q/80a4}{\texttt{80a4}} & $\Z/4\Z$\\ \hline
\texttt{[4X7,5S4]} & \href{http://www.lmfdb.org/EllipticCurve/Q/400h1}{\texttt{400h1}} & $\Z$\\ \hline
\texttt{[8X4,3Nn]} & \href{http://www.lmfdb.org/EllipticCurve/Q/576e3}{\texttt{576e3}} & $\Z/2\Z$\\ \hline
\texttt{[8X4,5B]} & \href{http://www.lmfdb.org/EllipticCurve/Q/320f4}{\texttt{320f4}} & $\Z/2\Z\oplus \Z$\\ \hline
\texttt{[8X4,5Nn]}  & \multicolumn{2}{c|}{Non Elliptic: $\mathcal E(\Q_2)=\emptyset$}  \\ \hline
\texttt{[8X4,13B]} & \href{http://www.lmfdb.org/EllipticCurve/Q/832h2}{\texttt{832h2}} & $\Z/2\Z\oplus \Z$\\ \hline
\texttt{[8X5,3Nn]} & \href{http://www.lmfdb.org/EllipticCurve/Q/576a3}{\texttt{576a3}} & $\Z/2\Z\oplus \Z$\\ \hline
\texttt{[8X5,5B]} & \href{http://www.lmfdb.org/EllipticCurve/Q/320c4}{\texttt{320c4}} & $\Z/2\Z$\\ \hline
\texttt{[8X5,5Nn]} & \href{http://www.lmfdb.org/EllipticCurve/Q/1600g2}{\texttt{1600g2}} & $\Z/2\Z\oplus \Z$\\ \hline
\texttt{[8X5,13B]} & \href{http://www.lmfdb.org/EllipticCurve/Q/832d2}{\texttt{832d2}} & $\Z/2\Z$\\ 
\hline 
\end{tabular}
\end{tabular}
\captionof{table}{Modular curves of genus 1}\label{tab:gen1nonCM}
\end{center}
}

\begin{center}
\begin{tabular}{|c|c|c|c|}
\hline 
Type & $j$-maps & $\mathcal E/\Q$ & Example \\
\hline
\texttt{[3Nn,5Nn]} & $\frac{125 \left(x y+4 x-y^2+4\right) \left(x y-12 x+2 y^2-8 y-12\right)^3 F_1(x,y)^3}{F_2(x,y)^5}$  &  \href{http://www.lmfdb.org/EllipticCurve/Q/225a1}{\texttt{225a1}} & Table \ref{ex3Nn}  \\
\hline
 \texttt{[3Nn,5Ns]}  & $-\frac{(y-2)^3 \left(y^2+y+4\right)^3 \left(y^2+6 y+4\right)^3}{\left(y^2+y-1\right)^5}$  &  \href{http://www.lmfdb.org/EllipticCurve/Q/225a1}{\texttt{225a1}} &Table \ref{ex3Nn}  \\
\hline
\texttt{[3Nn,7Nn]} & $\frac{\left(7 x y+28 x-2 y^2-30 y+59\right)^3 F_3(x,y)^3 F_4(x,y)^3 F_5(x,y)^3}{F_6(x,y)^7}$ &  \href{http://www.lmfdb.org/EllipticCurve/Q/441b1}{\texttt{441b1}} & Table \ref{ex3Nn} \\
\hline
\hline
\texttt{[4X7,5S4]} & $\frac{G_1(x,y)}{1024}$ & \href{http://www.lmfdb.org/EllipticCurve/Q/400h1}{\texttt{400h1}}  & \href{http://www.lmfdb.org/EllipticCurve/Q/12996c1}{\texttt{12996c1}} \\
\hline
\texttt{[8X4,5B]} & $-\frac{\left(x^2-470 x+5225\right)^3}{2 (x+15)^5}$ &  \href{http://www.lmfdb.org/EllipticCurve/Q/320f4}{\texttt{320f4}} & \href{http://www.lmfdb.org/EllipticCurve/Q/6400b1}{\texttt{6400b1}}\\
\hline
\texttt{[8X4,13B]}  & $-\frac{\left(x^2-2 x+28\right) \left(x^4-478 x^3+7688 x^2-38328 x+149808\right)^3}{2 (x+4)^{13}}$ &\href{http://www.lmfdb.org/EllipticCurve/Q/832h2}{\texttt{832h2}} & \href{http://www.lmfdb.org/EllipticCurve/Q/20736c1}{\texttt{20736c1}}\\
\hline
\texttt{[8X5,5Nn]} & $\frac{8000 G_2(x,y)}{\left(x^2-86 x-151\right)^{10}}$ & \href{http://www.lmfdb.org/EllipticCurve/Q/1600g2}{\texttt{1600g2}}&  \href{http://www.lmfdb.org/EllipticCurve/Q/313600bz1}{\texttt{313600bz1}}\\
\hline
\end{tabular}
\captionof{table}{Modular elliptic curve with positive rank}\label{tab:gen1PosRank}

\begin{tabular}{|c|c|c|}
\hline 
Type & Example & $\mathcal{N}_{\Q}(E)$  \\
\hline
\texttt{[3Nn,5Nn]} & $y^2 = x^3 + x^2 + 1218089157x + 10584902461321$  & 50840066816  \\
\hline
\texttt{[3Nn,5Ns]}  & $y^2 = x^3 - 1419330328x + 20580980954064$  &  1774432 \\
\hline
 \texttt{[3Nn,7Nn]} & $y^2 = x^3 - x^2 - 439163751869x + 112018153929262517$ &  1541679392 \\
\hline
\end{tabular}
\captionof{table}{Examples out of LMFDB: Genus $X_G =1$}\label{ex3Nn}
\end{center}

\begin{center}
\begin{tabular}{|c|c|c|c|}
\hline
Type & \multicolumn{2}{|c|}{$j$-invariants} & Examples \\
\hline
 \multirow{2}{*}{\texttt{[3B,5B]}} & \texttt{[3B,5B.4.1]} &   $\{-5\cdot 29^3/2^5,\,5\cdot 211^3/2^{15}\}$ &  \href{http://www.lmfdb.org/EllipticCurve/Q/400c3}{\texttt{400c3}},\, \href{http://www.lmfdb.org/EllipticCurve/Q/400c4}{\texttt{400c4}}  \\ 
\cline{2-4}
 & \texttt{[3B,5B.4.2]} & $\{-5^2/2,\,-5^2\cdot 241^3/2^3\}$   & \href{http://www.lmfdb.org/EllipticCurve/Q/400c1}{\texttt{400c1}} \,, \href{http://www.lmfdb.org/EllipticCurve/Q/400c2}{\texttt{400c2}}  \\ 
\hline
\texttt{[3B,5S4]} &\multicolumn{2}{|c|}{$ \{-2^4\cdot 3^2\cdot 13^3,\,2^4\cdot 3^3\}$} & \href{http://www.lmfdb.org/EllipticCurve/Q/1296i1}{\texttt{1296i1}},\,\href{http://www.lmfdb.org/EllipticCurve/Q/1296i2}{\texttt{1296i2}}  \\ 
\hline
\texttt{[3B,7B]} & \multicolumn{2}{|c|}{$\{-3^2\cdot  5^3\cdot  101^3/2^{21},\,-3^3\cdot  5^3\cdot  383^3/2^7,\,3^3\cdot   5^3/2,\,-3^2\cdot  5^6/2^3\}$} & \href{http://www.lmfdb.org/EllipticCurve/Q/1296k4}{\texttt{1296k4}},\,\href{http://www.lmfdb.org/EllipticCurve/Q/1296k3}{\texttt{1296k3}},\,\href{http://www.lmfdb.org/EllipticCurve/Q/1296k1}{\texttt{1296k1}},\,\href{http://www.lmfdb.org/EllipticCurve/Q/1296k2}{\texttt{1296k2}}  \\ 
\hline
\hline
\texttt{[3Ns,5B]} & \multicolumn{2}{|c|}{$\{11^3/2^3,\,-29^3\cdot 41^3/2^{15}\}$} & \href{http://www.lmfdb.org/EllipticCurve/Q/338e1}{\texttt{338e1}} ,\, \href{http://www.lmfdb.org/EllipticCurve/Q/338e2}{\texttt{338e2}} \\
\hline\hline
\texttt{[4X7,5B]} &  \multicolumn{2}{|c|}{$\{-5^2\cdot 41^3/2^2,\,5\cdot 59^3/2^{10}\}$}  &  \href{http://www.lmfdb.org/EllipticCurve/Q/14450bj1}{\texttt{14450bj1}},\, \href{http://www.lmfdb.org/EllipticCurve/Q/14450d1}{\texttt{14450d1}}  \\
\hline
\end{tabular}
\captionof{table}{Modular elliptic curves with rank 0: Coarse moduli}\label{tab:gen1r0}
\end{center}

{
\scriptsize
\begin{center}
\begin{tabular}{ccc}
\begin{tabular}{|c|c|}
\hline
Type  & All curves \\
\hline
\texttt{[3B,5B.1.1]} & \href{http://www.lmfdb.org/EllipticCurve/Q/50b1}{\texttt{50b1}},\href{http://www.lmfdb.org/EllipticCurve/Q/50b2}{\texttt{50b2} }  \\ 
\hline
\texttt{[3B,5B.1.2]} & \href{http://www.lmfdb.org/EllipticCurve/Q/50b3}{\texttt{50b3}},\href{http://www.lmfdb.org/EllipticCurve/Q/50b4}{\texttt{50b4}} \\ 
\hline
\texttt{[3B,7B.2.1]} & \href{http://www.lmfdb.org/EllipticCurve/Q/7938u3}{\texttt{7938u3}}, \href{http://www.lmfdb.org/EllipticCurve/Q/7938u4}{\texttt{7938u4}}  \\ 
\hline
\texttt{[3B,7B.2.3]} & \href{http://www.lmfdb.org/EllipticCurve/Q/7938u1}{\texttt{7938u1}}, \href{http://www.lmfdb.org/EllipticCurve/Q/7938u2}{\texttt{7938u2}}  \\ 
\hline
\end{tabular}
&
\begin{tabular}{|c|c|}
\hline
Type  & All curves \\
\hline
\texttt{[3B.1.1,5B.1.3]} & \href{http://www.lmfdb.org/EllipticCurve/Q/50a1}{\texttt{50a1}} \\ 
\hline
\texttt{[3B.1.1,5B.1.4]} & \href{http://www.lmfdb.org/EllipticCurve/Q/50a3}{\texttt{50a3}}  \\ 
\hline
\texttt{[3B.1.1,5B.4.1]} & \href{http://www.lmfdb.org/EllipticCurve/Q/450b4}{\texttt{450b4}} \\ 
\hline
\texttt{[3B.1.1,5B.4.2]} &\href{http://www.lmfdb.org/EllipticCurve/Q/450b2}{\texttt{450b2}}  \\ 
\hline
\texttt{[3B.1.1,5S4]} & \href{http://www.lmfdb.org/EllipticCurve/Q/324b1}{\texttt{324b1}},\href{http://www.lmfdb.org/EllipticCurve/Q/324d1}{\texttt{324d1}} \\ 
\hline
\texttt{[3B.1.1,7B]} &\href{http://www.lmfdb.org/EllipticCurve/Q/162c1}{\texttt{162c1}},\href{http://www.lmfdb.org/EllipticCurve/Q/162c3}{\texttt{162c3}}  \\ 
\hline
\texttt{[3B.1.1,7B.2.1]} & \href{http://www.lmfdb.org/EllipticCurve/Q/162b1}{\texttt{162b1}} \\ 
\hline
\texttt{[3B.1.1,7B.2.3]} & \href{http://www.lmfdb.org/EllipticCurve/Q/162b3}{\texttt{162b3}} \\ 
\hline
\end{tabular}
&
\begin{tabular}{|c|c|}
\hline
Type  & All curves \\
\hline
\texttt{[3B.1.2,5B.1.3]} & \href{http://www.lmfdb.org/EllipticCurve/Q/50a2}{\texttt{50a2}} \\ 
\hline
\texttt{[3B.1.2,5B.1.4]} & \href{http://www.lmfdb.org/EllipticCurve/Q/50a4}{\texttt{50a4}} \\ 
\hline
\texttt{[3B.1.2,5B.4.1]} & \href{http://www.lmfdb.org/EllipticCurve/Q/450b3}{\texttt{450b3}} \\ 
\hline
\texttt{[3B.1.2,5B.4.2]} & \href{http://www.lmfdb.org/EllipticCurve/Q/450b1}{\texttt{450b1}} \\ 
\hline
\texttt{[3B.1.2,5S4]} & \href{http://www.lmfdb.org/EllipticCurve/Q/324b2}{\texttt{324b2}},\href{http://www.lmfdb.org/EllipticCurve/Q/324d2}{\texttt{324d2}} \\ 
\hline
\texttt{[3B.1.2,7B]} & \href{http://www.lmfdb.org/EllipticCurve/Q/162c2}{\texttt{162c2}}, \href{http://www.lmfdb.org/EllipticCurve/Q/162c4}{\texttt{162c4}} \\ 
\hline
\texttt{[3B.1.2,7B.2.1]} & \href{http://www.lmfdb.org/EllipticCurve/Q/162b2}{\texttt{162b2}} \\ 
\hline
\texttt{[3B.1.2,7B.2.3]} &  \href{http://www.lmfdb.org/EllipticCurve/Q/162b4}{\texttt{162b4}} \\ 
\hline
\end{tabular}
\end{tabular}
\captionof{table}{Modular elliptic curves with rank 0: Fine moduli}\label{tab:gen1r0fine}
\end{center}
}


{\small
\begin{center}
\begin{tabular}{|c|c|c|c|}
\hline
Type & Rank & Non-cuspidal points & $j$-invariants \\
\hline
\texttt{[2B,5Ns.2.1]} & 0 & $\emptyset$ & $-$ \\ 
 \hline
\texttt{[2B,7Nn]} & 0  &  $( -4, 3 ),( 1/2, 3 )$ & $2^6 3^3$ \\ 
& & $( -1, 0 )$ & 0 \\ 
& & $( 32, -1 )$ & $2^3 3^3 11^3$ \\ 
 \hline
\texttt{[2Cn,5Ns]} & 0  &  $( 0, -2 )$ & $2^6 3^3$ \\ 
 \hline
\texttt{[2Cn,7B.6.1]} & 0 & $\emptyset$ & $-$ \\ 
 \hline
\texttt{[2Cn,7B.6.2]} & 0 & $\emptyset$ & $-$ \\ 
 \hline
\texttt{[2Cn,7B.6.3]} & 0 & $\emptyset$ & $-$ \\ 
 \hline
\texttt{[2Cs,5S4]} & 0  &  $( 1, 3 ),( -2, 3 ),( -1/2, 3 )$ & $2^6 3^3$ \\ 
 \hline
\texttt{[3B,5Nn]} & 0  &  $( 1, 1 )$ & $2^4 3^3 5^3$ \\ 
& & $( -1/9, 1/2 )$ & $-2^{15} 3\, 5^3$ \\ 
& & $( -9, -1/2 ),( -9, -1 )$ & 0 \\ 
& & $( -1, -1/2 ),( -1, -1 )$ & \\ 
 \hline
\texttt{[3Nn,13B]} & 0 & $\emptyset$ & $-$ \\ 
 \hline
\texttt{[3Nn,5B.4.1]} & 0 & $\emptyset$ & $-$ \\ 
 \hline
\texttt{[3Nn,5B.4.2]} & 0 & $\emptyset$ & $-$ \\ 
 \hline
\texttt{[3Nn,7B]} & 0  &  $( 255, 7 )$ & $3^3 5^3 17^3$ \\ 
& & $( -15, -7 )$ & $-3^3 5^3$ \\ 
 \hline
\texttt{[3Ns,5S4]} & 1  &  $( 1/3, -8 ),( -9, -8 )$ & $-2^{15}$ \\ 
& & $( -1, 0 ),( 3, 0 )$ & 0 \\ 
\hline
\texttt{[5S4,7B]} & 0 & $\emptyset$ & $-$ \\ 
 \hline
 \hline
  \texttt{[2Cn,9XE]} & 0 & $\emptyset$ & $-$\\ 
 \hline
\texttt{[4X3,9XE]} & 0 & $\emptyset$ & $-$\\ 
 \hline
\texttt{[4X7,5Nn]} & 2  & $(40 , 1/2)$ &  $-2^{15}3\,5^3$\\
& & $(-440 , 3/5)$ & $-2^{15}3^3 5^3 11^3$\\
& & $(120 , -3/2)$ & $-2^{18}3^3 5^3$\\
& & $(-16008 , -21/13)$ & $-2^{18} 3^3 5^3 23^3 29^3$\\
&& $(0 , -1),(0 , -1/2)$  & 0\\
& & $(-8 , -1/2 ) ,(-8 , -1)$ & \\
\hline
\texttt{[4X7,7B]} & 1  &  $( -3/2, -49/4 )$ & $3^3 13/2^2$ \\ 
& & $( -479/16, -4 )$ & $-3^3 13\, 479^3/2^{14}$ \\ 
 \hline \rowcolor{white}
\texttt{[8X4,9XE]} & 0 & $\emptyset$ & $-$ \\ 
 \hline
\texttt{[8X5,9XE]} & 2 & $C(\Q_3)=\emptyset$ & $-$   \\
\hline
\end{tabular}
\captionof{table}{Modular curves of genus 2}\label{tab:g2}
\end{center}
}


{\footnotesize
\begin{center}
\begin{tabular}{|l|l|l|l|l|l|l|}
\hline
$\texttt{[2B,7Ns]}_{3}^\text{cm}$ &  $\texttt{[2Cn,7Nn]}_{3}^\text{cm}$ &  $\texttt{[2Cn,7Ns]}_{3}^\emptyset$ &  $\texttt{[3Nn,7Ns]}_{3}^\text{cm}$ &  $\texttt{[3Ns,5Nn]}_{3}^\text{cm}$ &  $\texttt{[4X3,7Nn]}_{3}^\text{cm}$ & $\texttt{[4X3,7Ns]}_{3}^\emptyset$ \\
 $\texttt{[4X7,13B]}_{3}^\emptyset$ &  $\texttt{[8X4,7Nn]}_{3}^\text{cm}$ &  $\texttt{[8X4,7Ns]}_{3}^\emptyset$ & $\texttt{[8X5,7Nn]}_{3}^\text{cm}$ &  $\texttt{[8X5,7Ns]}_{3}$ &  $\texttt{[3Nn,5Cs]}_{4}^\emptyset$ &   $\texttt{[3Nn,5Ns.2.1]}_{4}^\text{cm}$ \\
 $\texttt{[5S4,13B]}_{4}^\emptyset$ &  $\texttt{[4X7,7Nn]}_{4}^\text{cm}$ &  $\texttt{[9XE,2B]}_{4}^\text{cm}$ & $\texttt{[3B,7Nn]}_{5}^\text{cm}$ &  $\texttt{[5Nn,7B]}_{5}^\text{cm}$ &  $\texttt{[3B,7Ns]}_{6}^\text{cm}$  &
 $\texttt{[3Ns,7Nn]}_{6}^\text{cm}$ \\
  $\texttt{[5B,7Nn]}_{6}^\emptyset$ &  $\texttt{[4X7,7Ns]}_{6}^\text{cm}$ & $\texttt{[4X7,9XE]}_{6}$ &  $\texttt{[2Cn,11Nn]}_{7}^\text{cm}$ &  $\texttt{[3Nn,11Nn]}_{7}^\text{cm}$ &  $\texttt{[5S4,7Nn]}_{7}^\text{cm}$ &
  $\texttt{[4X3,11Nn]}_{7}^\text{cm}$ \\
    $\texttt{[8X4,11Nn]}_{7}^\text{cm}$ & $\texttt{[8X5,11Nn]}_{7}^\text{cm}$ &  $\texttt{[2B,11Nn]}_{8}^\text{cm}$ &  $\texttt{[9XE,5S4]}_{8}^\text{cm}$ &  $\texttt{[5B,7Ns]}_{9}^\emptyset$ &  $\texttt{[5Nn,13B]}_{9}^\emptyset$ & 
 $\texttt{[5S4,7Ns]}_{9}^\text{cm}$ \\$\texttt{[9XE,5B]}_{10}^\emptyset$ &  $\texttt{[5Nn,7Nn]}_{13}^\text{cm}$ &  $\texttt{[4X7,11Nn]}_{13}^\text{cm}$ &  $\texttt{[3B,11Nn]}_{14}^\text{cm}$ & $\texttt{[9XE,7B]}_{14}^\emptyset$ &  $\texttt{[5Nn,7Ns]}_{18}^\text{cm}$ &   $\texttt{[9XE,5Nn]}_{18}^\text{cm}$ \\
   $\texttt{[5S4,11Nn]}_{19}^\text{cm}$ &  $\texttt{[5B,11Nn]}_{20}^\emptyset$ &  $\texttt{[7Nn,13B]}_{20}^\emptyset$ & $\texttt{[9XE,13B]}_{26}^\emptyset$ &  $\texttt{[7Ns,13B]}_{27}^\emptyset$ & $\texttt{[7B,11Nn]}_{32}^\emptyset$ & $\texttt{[5Nn,11Nn]}_{38}^\text{cm}$ \\   $\texttt{[9XE,7Nn]}_{40}^\text{cm}$ & $\texttt{[9XE,7Ns]}_{54}^\text{cm}$ &$\texttt{[11Nn,13B]}_{56}^\emptyset$  & $\texttt{[7Nn,11Nn]}_{81}^\text{cm}$ &  $\texttt{[9XE,11Nn]}_{111}^\text{cm}$ &   $\texttt{[7Ns,11Nn]}_{112}^\text{cm}$   & \\
 \hline
\end{tabular}
\captionof{table}{Remaining maximal modular curves of genus > 2}\label{rema}
 \end{center}
 }

\begin{center}
\begin{tabular}{|c|c|c|c|}
\hline 
Type & $j$-invariant & Example  & Conductor  \\
\hline
\texttt{[4X7,9XE]}
& $-2^2 3^7 5^3 439^3$ & $y^2 = x^3 - 1126035x + 459913278$ & 701784\\
\hline
 \multirow{2}{*}{\texttt{13S4}} & \footnotesize$-\frac{2^{12}\cdot 5^3\cdot 11\cdot 13^4}{3^{13}}$&  \scriptsize $y^2 + y = x^3 + x^2 - 7653878762768x + 8080142566037338385$ & \scriptsize  374369283576145574257827 \\
\cline{2-4}
&  \footnotesize $\frac{2^{18}\cdot3^3\cdot 13^4\cdot 127^3\cdot 139^3\cdot 157^3\cdot 283^3\cdot 929}{5^{13}\cdot 61^{13}}$ & \scriptsize $y^2 + y = x^3 - 53690013976669148x + 4788368560731534924873003$ & \scriptsize 528531611786945 \\
\hline
\end{tabular}
\captionof{table}{Examples out of LMFDB: Genus $X_G >1$}\label{adicex}
\end{center}

\noindent{\bf Extra information: Polynomials for some of the tables:}

\

\noindent$\bullet$ Table \ref{jmaps}}:
$$
\scriptsize
\begin{array}{rl}
P_{11}(x,y)=&(x^2+3x-6)(11(x^2-5)y+(2x^4+23x^3-72x^2-28x+127))(6y+11x-19)(22(x-2)y+(5x^3+17x^2-112x+120))\\
P_{13}(t)=&t^{12} + 231t^{11} + 269t^{10} - 3160t^9 + 6022t^8 - 9616t^7 + 21880t^6 - 34102t^5 + 28297t^4 - 12455t^3 + 2876t^2 - 243t + 1\\[2mm]
\end{array}
$$
\noindent $\bullet$ Table \ref{Eq_finemoduli}}:
$$\scriptsize
\begin{array}{rl}
Q_{7,3}(t)=&t^{12} - 18t^{11} + 117t^{10} - 354t^9 + 570t^8 - 486t^7+ 273t^6 - 222t^5  + 174t^4 - 46t^3 - 15t^2 + 6t + 1\\
Q_{7,4}(t)=&t^{12}-522 t^{11}-8955 t^{10}+37950 t^9-70998 t^8+131562 t^7 -253239 t^6+316290 t^5-218058 t^4+80090 t^3-14631 t^2+510 t+1\\
Q_{7,5}(t)=&  (t^4 - 6t^3 + 17t^2 - 24t + 9)(3t^4 - 4t^3 - 5t^2 - 2t - 1)(9t^4 - 12t^3 - t^2 + 8t - 3)\\
P_{13,4}(t)=& t^8 + 235t^7 + 1207t^6 + 955t^5 + 3840t^4 - 955t^3 + 1207t^2 - 235t+ 1\\
P_{13,5}(t)=& t^8 - 5t^7 + 7t^6 - 5t^5 + 5t^3 + 7t^2 + 5t + 1\\
Q_{13,4}(t)=&t^{12} - 512 t^{11} - 13079 t^{10} - 32300 t^9 - 104792 t^8 - 111870 t^7 - 419368 t^6 + 111870 t^5 - 104792 t^4 + 32300 t^3 - 13079 t^2 + 512 t +1 \\
Q_{13,5}(t)=&t^{12} - 8 t^{11} + 25 t^{10} - 44 t^9 + 40 t^8 + 18 t^7 - 40 t^6 - 18 t^5 + 40 t^4 + 44 t^3 + 25 t^2 + 8 t + 1\\[2mm]
\end{array}
$$
\noindent $\bullet$ Table \ref{jmapsj0coarse}:
$$\scriptsize
\begin{array}{rl}
P_1(t)=&9289670605927434230887788667927350765223936 t^{12}+29278270369999901950955093380872504213504 t^{11}\\
 & +42292791583476109342488555094120464384 t^{10}+37025228725171770917082043542364160 t^9\\
 & +21878993767277277183859380817920 t^8+9193668584402084086752989184 t^7+2816901155195900104390656 t^6\\
 & +634096368731743520256 t^5+104078307564875520 t^4+12147786424640 t^3+957045024 t^2+45696 t+1\\
P_2(t)  = & 65536 t^{12}-4063232 t^{11}+16777216 t^{10}-28958720 t^9+27832320 t^8-16576512 t^7\\
 & +6385664 t^6-1608192 t^5+261120 t^4-25920 t^3+1376 t^2-32 t+1\\[1mm]
 \end{array}
$$
\noindent $\bullet$ Table \ref{tab:gen1PosRank}:
$$\scriptsize
\begin{array}{rl}
F_1(x,y)=&53 x^2 y^2-232 x^2 y+272 x^2+17 x y^3-372 x y^2+504 x y+544 x+2 y^4-56 y^3+468 y^2+736 y+272\\
F_2(x,y)=&11 x^2 y^2-24 x^2 y-16 x^2-x y^3-44 x y^2+8 x y-32 x-y^4+8 y^3+76 y^2+32 y-16\\
F_3(x,y)=&28 x^2 y^2-266 x^2 y-1169 x^2+14 x y^3-126 x y^2+378 x y-5180 x-5 y^4-52 y^3+1032 y^2+2612 y-5711\\
F_4(x,y)=&28 x^2 y^2+77 x^2 y+203 x^2-7 x y^3-280 x y^2-189 x y+875 x+y^4+30 y^3+509 y^2-542 y+956\\
F_5(x,y)=&14 x^2 y^2-133 x^2 y+616 x^2+7 x y^3-210 x y^2+385 x y+2947 x+2 y^4-38 y^3+185 y^2+1758 y+3529\\
F_6(x,y)=&7 x^3 y^3+329 x^3 y^2-448 x^3 y-4207 x^3-21 x^2 y^4-42 x^2 y^3-1575 x^2 y^2-1554 x^2 y-28749 x^2+147 x y^4+\\
& 294 x y^3+3822 x y^2+3675 x y-65268 x+y^6-4 y^5-565 y^4-260 y^3+20710 y^2+11170 y-49223\\
G_1(x,y)=&-x^{10}-210 x^9-20 x^8 y-7085 x^8-1440 x^7 y-72280 x^7-25904 x^6 y-262770 x^6-150880 x^5 y-650220 x^5-  323320 x^4 y\\
& -2073650 x^4-703840 x^3 y-3299800 x^3-1393200 x^2 y-4157325 x^2-1015200 x y  -4799250 x-580500 y-1454625\\
G_2(x,y)=&x^{20}+1260 x^{19}-56 x^{18} y+42870 x^{18}-13456 x^{17} y+403100 x^{17}+440168 x^{16} y+1335852085 x^{16}\\
 &  -121312128 x^{15} y-61881358352 x^{15}-1307723360 x^{14} y+6615621682760 x^{14}+89306042944 x^{13} y\\
 &  -160467247908880 x^{13}-47113584888416 x^{12} y+4865894105161650 x^{12}+2738042048551808 x^{11} y\\
 &  +387841608615974760 x^{11}-171889581774911248 x^{10} y-31426480449116277436 x^{10}+\\
 &  6725955511078796960 x^9 y+1497426397985497774920 x^9-223160321487761337552 x^8 y-\\
 &  47343303309980334099710 x^8+4384140813703797518208 x^7 y+1177623302292232438580400 x^7+\\
 &  23859921698214377667616 x^6 y-18020700523192267143943480 x^6-3625628171475452791902656 x^5 y-\\
 &  91886854155885171733697488 x^5+78452198711443613319043360 x^4 y+\\
 &  5946599171665206805434677005 x^4- 1229423961196759128856211328 x^3 y-\\
 &  57315101050144284752075165940 x^3+3544228211661585891755479432 x^2 y+\\
 &  2147304967678389238737065757750 x^2+693493926872718537524755967344 x y+\\
 &  51452457690608652712058018333180 x+ 4183073553838029267128981247656 y+\\
 &  223823753822802307667379753623561
\end{array}
$$

\bibliographystyle{plain}

\end{document}